\documentclass[11pt,letterpaper]{article}

\usepackage[margin=0.9in]{geometry}
\usepackage{amsmath,amssymb,amsthm}
\usepackage{graphicx}
\usepackage[authoryear,longnamesfirst]{natbib}
\setcitestyle{authoryear,round,semicolon,aysep={,},yysep={,}}
\usepackage[colorlinks=true,linkcolor=blue,citecolor=blue,urlcolor=blue,
  pdftitle={Consistency of penalized maximum likelihood estimation for multidimensional change-in-velocity detection},
  pdfauthor={Linh Do, Dat Do, Scott A. McKinley}]{hyperref}

\newcommand{\Rbb}{\mathbb{R}}
\newcommand{\Pbb}{\mathbb{P}}
\newcommand{\tbf}{\boldsymbol{\tau}}
\newcommand{\Vbf}{\mathbf{V}}
\newcommand{\av}{\underline{a}}
\newcommand{\norm}[1]{\left\|#1\right\|}
\DeclareMathOperator*{\argmax}{arg\,max}

\newtheorem{theorem}{Theorem}
\newtheorem{lemma}[theorem]{Lemma}

\theoremstyle{definition}
\newtheorem{definition}{Definition}
\newtheorem{remark}{Remark}
\theoremstyle{plain}
\newtheorem*{restatedlemma}{Lemma}
\newtheorem*{restatedtheorem}{Theorem}
\theoremstyle{definition}

\renewenvironment{proof}{\noindent\textit{Proof.}}{\hfill$\square$\medskip}

\title{Consistency of penalized maximum likelihood estimation
  for multidimensional change-in-velocity detection\thanks{This research was supported by the NSF-Simons Southeast Center for
Mathematics and Biology (SCMB) through grants NSF DMS-1764406 and Simons
Foundation/SFARI 594594.}}

\author{%
  Linh Do$^{1,*}$ \qquad Dat Do$^{2}$ \qquad Scott A.\ McKinley$^{3}$\\[1.5ex]
  {\small $^{1}$Department of Biostatistics and Bioinformatics,
    Duke University School of Medicine, Durham, NC, USA}\\
  {\small $^{2}$Department of Statistics,
    University of Chicago, Chicago, IL, USA}\\
  {\small $^{3}$Department of Mathematics,
    Tulane University, New Orleans, LA, USA}\\[1ex]
  {\small $^{*}$Corresponding author: thuylinh.do@duke.edu}\\
  {\small dodat@uchicago.edu \quad scott.mckinley@tulane.edu}
}
\date{}

\begin{document}
\maketitle

\begin{abstract}
We establish statistical consistency for the penalized maximum-likelihood
estimator underlying CPLASS, a method for detecting changes in velocity in
\(d\)-dimensional time-series data. The signal is modeled as a continuous
piecewise-linear trajectory observed with independent Gaussian noise. Unlike
classical change-in-mean models, continuity across adjacent segments couples
their parameters and prevents direct application of standard segmentation
arguments. Under compact parameter spaces, minimum segment-length and
velocity-jump conditions, and a strengthened Schwarz information criterion
penalty \(\rho_k(\log n)^\gamma\) with \(\gamma>1\), we prove joint consistency
of the estimated number of segments and all changepoint locations in the
high-frequency regime \(n\to\infty\). The maximum changepoint-location error is
\(O_{\mathbb P}\{(\log n/n)^{1/2}\}\). The proof first uses empirical-process
theory to establish convergence of the fitted signal, variance, and likelihood
for each fixed model size. It then combines overfitting and underfitting
arguments with a two-stage geometric localization analysis, yielding an initial
\((\log n/n)^{1/3}\) rate that is sharpened by exploiting the local two-segment
continuous piecewise-linear structure.
\end{abstract}


\medskip
\section{Introduction}

Multiple changepoint detection for piecewise-constant signals is classical,
with well-developed dynamic programming, pruning, and greedy segmentation
methods
\citep{binseg,Optimal_Partition,PELT,Piotr2014WBS}, together with a mature
consistency theory
\citep{YAO1988181,Yao1989}. Change-in-velocity, or equivalently
change-in-slope in 1D case, detection appears similar because the observation interval is
again partitioned into segments and a simple model is fitted within each
segment. The continuity constraint linking adjacent line segments, however,
changes both the computational and statistical structure of the problem.

For a change-in-mean model, the residual cost of a segment depends only on
the observations within that segment and its boundary changepoints. This
separability produces an additive objective and supports standard
dynamic-programming recursions and one-change-at-a-time procedures. In a
continuous piecewise-linear model, the fitted value at the end of one segment
becomes the starting value of the next. Moving a changepoint can therefore
alter the slopes and intercepts of several neighboring segments. As a
consequence, the best one-changepoint approximation to a signal with multiple
true slope changes need not place its changepoint near any individual true
change
\citep{cpop,BaranowskiFryzlewicz2019,Do2025-vv}. Standard binary
segmentation and unmodified segment-additive dynamic programming therefore do
not transfer directly. CPOP \citep{cpop}, for example, restores a
dynamic-programming recursion by augmenting the state with the fitted endpoint
value inherited from the preceding segment.

Several specialized approaches have been developed for one-dimensional
change-in-slope problems. Trend filtering \citep{trendfilter} replaces the
$\ell_0$ penalty on slope changes with a total-variation penalty on discrete
derivatives. Narrowest-Over-Threshold
\citep{BaranowskiFryzlewicz2019} searches over many local intervals and uses
strengthened logarithmic control to localize slope changes. CPOP
\citep{cpop} retains an $\ell_0$-penalized least-squares formulation but
modifies dynamic programming to accommodate continuity. Narrowest
Significance Pursuit \citep{Fryzlewicz2023} instead returns short intervals
that contain changepoints with prescribed global significance. The available
theory for these procedures is primarily developed for scalar-valued response
sequences. Although some frameworks admit possible multivariate extensions,
the corresponding consistency arguments for vector-valued continuous
piecewise-linear signals have not been developed in detail.

The present work studies a $d$-dimensional continuous piecewise-linear signal
with a common set of changepoints across coordinates. Our estimator is a
global penalized maximum-likelihood estimator using a strengthened Schwarz
information criterion penalty
\(
  \rho_k(\log n)^\gamma,
  \gamma>1.
\)
Strengthened logarithmic penalties have previously been used to control the
stochastic complexity generated by searches over growing collections of
candidate changepoints or intervals
\citep{fryzlewicz2014segmentation,Piotr2014WBS,
BaranowskiFryzlewicz2019}. Here, the same device separates the likelihood gain
available to overfitted continuous piecewise-linear models from the additional
penalty incurred by introducing spurious segments.

The technical foundation of the proof is the empirical-process theory for
maximum-likelihood and other $M$-estimators developed by
\citet{Vandegeer}. A Hellinger bracketing-entropy bound for the
$d$-dimensional continuous piecewise-linear model yields high-probability
rates for the fitted signal, variance, and likelihood for each fixed number
of segments. Model-selection consistency is then obtained by treating
overfitting and underfitting separately. Overfitting is excluded because the
additional sSIC penalty dominates the stochastic likelihood improvement,
whereas underfitting is excluded because merging across a true velocity
change creates a non-vanishing approximation error under minimum spacing and
minimum jump-size conditions.

The localization argument is related to segmented-regression consistency
results such as \citet{Liu1997-go}, but is adapted here to the
multidimensional, continuity-constrained, high-frequency setting. It proceeds
in two stages. A best linear approximation on a shrinking neighborhood of a
true changepoint first yields the preliminary rate
$(\log n/n)^{1/3}$. Enlarging the neighborhood to include an adjacent segment
and analyzing the best continuous two-segment approximation sharpens this to
$(\log n/n)^{1/2}$.

This paper makes rigorous the consistency claim associated with the global
estimator underlying CPLASS \citep{Do2025-vv}. We prove joint consistency
of the estimated number of segments and all changepoint locations for
$d$-dimensional continuous piecewise-linear trajectories observed in the
high-frequency regime. The result extends related scalar change-in-slope and
segmented-regression arguments to vector-valued velocities under a shared
continuity constraint.

\section{Model and estimator}\label{sec:model}

We observe $(Y_i)_{i=1}^n\subset\Rbb^d$ at the fixed design points
\(
  t_i=\frac{i}{n}, i=1,\ldots,n,
\)
according to
\begin{equation}\label{eq:true-model}
  Y_i \overset{\mathrm{ind.}}{\sim}
  \mathcal{N}_d\!\left(f^0(t_i),\,\sigma_0^2 I_d\right),
   i=1,\ldots,n.
\end{equation}
Thus, the asymptotic regime considered throughout is the high-frequency
limit $n\to\infty$ on the fixed time interval $[0,1]$. Throughout, the dimension $d$ and the true number of segments $k_0$ are
fixed and do not depend on $n$.

For an integer $k\ge 1$, let $\mathcal{F}_k$ denote the class of
$d$-dimensional continuous piecewise-linear functions that admit a
representation with $k$ linear segments and changepoints on the observation
grid
\(
  \mathcal{T}_n
  :=
  \left\{\frac{1}{n},\ldots,\frac{n-1}{n}\right\}.
\)
For notational simplicity, we suppress the dependence of $\mathcal{F}_k$
on $n$ induced by the observation grid.

For simplicity, we assume throughout that the true changepoints are aligned
with the observation grid, so that
$\tau_j^0\in\mathcal{T}_n, j=1,\ldots,k_0-1$,
and hence $f^0\in\mathcal F_{k_0}$. More precisely, the asymptotic result is
understood along a sequence of observation grids containing the true
changepoints. This assumption avoids an inessential discretization issue.
For an off-grid true changepoint, one may instead use its nearest point in
$\mathcal T_n$; under the bounded-velocity assumption, the resulting signal
approximation has empirical squared error of order $O(n^{-2})$, which is
negligible relative to the estimation rate established below.

Such a representation is parametrized by the ordered changepoint vector
\(
  \tbf=(\tau_1,\ldots,\tau_{k-1}),
  0=:\tau_0<\tau_1<\cdots<\tau_{k-1}<\tau_k:=1,
\)
with $\tau_j\in\mathcal{T}_n$, velocity vectors
\(
  \Vbf=(V_1,\ldots,V_k)\in\mathcal{V}^k,
  \mathcal{V}\subset\Rbb^d,
\)
and an initial position
\(
  \av\in\mathcal{A}\subset\Rbb^d.
\)
The sets $\mathcal{V}$ and $\mathcal{A}$ are compact, and the noise
variance belongs to
\(
  \Omega=
  [\underline{\sigma}^2,\overline{\sigma}^2]
  \subset(0,\infty).
\)

Writing $(x)_+=x\mathbf{1}\{x>0\}$, the signal associated with
$(\tbf,\Vbf,\av)$ is
\begin{equation}\label{eq:signal}
  f_{\tbf,\Vbf,\av}(t)
  =
  \av+V_1t+
  \sum_{j=1}^{k-1}
  (V_{j+1}-V_j)(t-\tau_j)_+,
   t\in[0,1].
\end{equation}
Equivalently, for $t\in[\tau_{i-1},\tau_i]$,
$i=1,\ldots,k$,
\begin{equation}\label{eq:signal-piecewise}
  f_{\tbf,\Vbf,\av}(t)
  =
  \av+V_i t
  -
  \sum_{j=1}^{i-1}
  (V_{j+1}-V_j)\tau_j,
\end{equation}
with the convention that an empty sum equals zero. Both representations
ensure that adjacent linear pieces meet continuously at every changepoint.

The candidate classes are defined by representability rather than
minimality: adjacent velocities in a representation belonging to
$\mathcal{F}_k$ are permitted to coincide. Consequently, the model classes
are nested in $k$. In contrast, the true signal
$f^0=f_{\tbf^0,\Vbf^0,\av^0}\in\mathcal F_{k_0}$ is assumed to have a
minimal representation with $k_0$ segments, so that
\(
  V_{j+1}^0\neq V_j^0, j=1,\ldots,k_0-1.
\)

Given a fixed upper bound $\overline{k}\ge k_0$ on the number of segments,
with $\overline{k}$ independent of $n$, we study the global penalized
maximum-likelihood estimator

\begin{equation}\label{eq:pen-MLE}
  (\widehat{f}_n,\widehat{\sigma}_n^2,\widehat{k}_n)
  =
  \argmax_{\substack{1\le k\le\overline{k}\\
                     f\in\mathcal{F}_k,\;\sigma^2\in\Omega}}
  \left\{
  2\sum_{i=1}^n
  \log\mathcal{N}_d\!\left(
    Y_i\mid f(t_i),\sigma^2 I_d
  \right)
  -
  \rho_k(\log n)^\gamma
  \right\},
   \gamma>1,
\end{equation}
where $\rho_k:=d(k+1)+k.$

Here, $\rho_k$ counts the $kd$ velocity coordinates, the $d$
coordinates of the initial position, the $k-1$ changepoint locations,
and the variance parameter:
\(
  kd+d+(k-1)+1=d(k+1)+k.
\)

Equivalently, for each fixed $k\in[\overline{k}]$, define the
unpenalized maximum-likelihood estimator
\begin{equation}\label{eq:MLE-k}
  (\widehat{f}_n^{(k)},\widehat{\sigma}_{n,k}^2)
  =
  \argmax_{f\in\mathcal{F}_k,\;\sigma^2\in\Omega}
  \sum_{i=1}^n
  \log\mathcal{N}_d\!\left(
    Y_i\mid f(t_i),\sigma^2 I_d
  \right).
\end{equation}
The estimated number of segments is then selected according to
\begin{equation}\label{eq:khat}
  \widehat{k}_n
  =
  \argmax_{1\le k\le\overline{k}}
  \left\{
  2\sum_{i=1}^n
  \log\mathcal{N}_d\!\left(
    Y_i\mid \widehat{f}_n^{(k)}(t_i),
    \widehat{\sigma}_{n,k}^2 I_d
  \right)
  -
  \rho_k(\log n)^\gamma
  \right\}.
\end{equation}

The compact variance space
\(
  \Omega=
  [\underline{\sigma}^2,\overline{\sigma}^2]
  \subset(0,\infty)
\)
is used as a regularity condition in the consistency argument, and we
assume that the true variance satisfies
\(
  \sigma_0^2
  \in
  (\underline{\sigma}^2,\overline{\sigma}^2).
\)

\section{Main results}\label{sec:results}

For $f\in\mathcal{F}_k$ and $\sigma^2\in\Omega$, define the normalized
log-likelihood
\begin{equation}\label{eq:normalized-loglik}
  \overline{\mathcal{L}}_n(f,\sigma^2)
  :=
  \frac{1}{n}
  \sum_{i=1}^n
  \log\mathcal{N}_d\!\left(
    Y_i\mid f(t_i),\sigma^2 I_d
  \right),
\end{equation}
and, for functions $f,g:[0,1]\to\Rbb^d$, define the empirical squared
norm
\begin{equation}\label{eq:empirical-norm}
  \|f-g\|_n^2
  :=
  \frac{1}{n}
  \sum_{i=1}^n
  \left\|f(t_i)-g(t_i)\right\|_2^2.
\end{equation}

For $f,g:[0,1]\to\Rbb^d$ and
$\sigma^2,\widetilde{\sigma}^2\in\Omega$, write
\(
  p_{f,\sigma^2,i}
  :=
  \mathcal N_d\!\left(
    f(t_i),\sigma^2I_d
  \right),
  p_{g,\widetilde{\sigma}^2,i}
  :=
  \mathcal N_d\!\left(
    g(t_i),\widetilde{\sigma}^2I_d
  \right),
\)
and define the average squared Hellinger distance
\begin{equation}\label{eq:average-hellinger}
  h_n^2\!\left(
    p_{f,\sigma^2},
    p_{g,\widetilde{\sigma}^2}
  \right)
  :=
  \frac1n
  \sum_{i=1}^n
  h^2\!\left(
    p_{f,\sigma^2,i},
    p_{g,\widetilde{\sigma}^2,i}
  \right).
\end{equation}

Recall that
\(
  \rho_k=d(k+1)+k.
\)
Because the candidate classes are defined by representability,
\(
  \mathcal{F}_{k_0}\subseteq\mathcal{F}_k,
   k\ge k_0.
\)

\begin{lemma}[Gaussian Hellinger equivalence]
\label{lem:gaussian-hellinger}
Let
\(
  p_{\mu,\sigma^2}
  =
  \mathcal N_d(\mu,\sigma^2I_d), 
  p_{\widetilde\mu,\widetilde\sigma^2}
  =
  \mathcal N_d(\widetilde\mu,\widetilde\sigma^2I_d).
\)
Suppose that $\mu$ and $\widetilde\mu$ range over a fixed compact subset
of $\Rbb^d$ and that
$\sigma^2,\widetilde\sigma^2\in\Omega$. Then
\begin{equation}\label{eq:gaussian-hellinger-equivalence}
  h^2\!\left(
    p_{\mu,\sigma^2},
    p_{\widetilde\mu,\widetilde\sigma^2}
  \right)
  \asymp
  \|\mu-\widetilde\mu\|_2^2
  +
  \left|\sigma^2-\widetilde\sigma^2\right|^2,
\end{equation}
where the equivalence constants depend only on the compact parameter
spaces.

Consequently, for $f,g\in\mathcal F_k$ and
$\sigma^2,\widetilde\sigma^2\in\Omega$,
\begin{equation}\label{eq:average-hellinger-equivalence}
  h_n^2\!\left(
    p_{f,\sigma^2},
    p_{g,\widetilde\sigma^2}
  \right)
  \asymp
  \|f-g\|_n^2
  +
  \left|\sigma^2-\widetilde\sigma^2\right|^2.
\end{equation}
\end{lemma}

The compactness of the relevant mean space follows from the compactness of
$\mathcal A$ and $\mathcal V$ together with $t\in[0,1]$ and the fixed upper
bound on the number of segments.

\begin{lemma}[Fixed-model convergence rates]\label{lem:rates}
Suppose that the data are generated according to
\eqref{eq:true-model}, with
\(
  f^0\in\mathcal{F}_{k_0},\sigma_0^2\in
(\underline{\sigma}^2,\overline{\sigma}^2).
\)
Then, for each fixed $k\ge k_0$, there exist constants
$C,c_1,c_2>0$, not depending on $n$, such that, with probability at
least $1-c_1n^{-c_2}$,
\begin{align}
  \left\|\widehat{f}_n^{(k)}-f^0\right\|_n^2
  &\le
  C\frac{\rho_k\log n}{n},
  \label{eq:fn-rate}\\
  \left|\widehat{\sigma}_{n,k}^2-\sigma_0^2\right|
  &\le
  C\left(
    \frac{\rho_k\log n}{n}
  \right)^{1/2},
  \label{eq:sigma-rate}\\
  0
  \le
  \overline{\mathcal{L}}_n\!\left(
    \widehat{f}_n^{(k)},
    \widehat{\sigma}_{n,k}^2
  \right)
  -
  \overline{\mathcal{L}}_n(f^0,\sigma_0^2)
  &\le
  C\frac{\rho_k\log n}{n}.
  \label{eq:llh-rate}
\end{align}
The constants may depend on
$d$, $k$, $\mathcal{V}$, $\mathcal{A}$, and $\Omega$, but not on $n$.
\end{lemma}

We assume that the true representation is well separated: there exist
constants $\underline C_1,\underline C_2>0$ such that
\begin{equation}
\label{eq:min-spacing}
\min_{1\leq j\leq k_0}
\left(
\tau_j^0-\tau_{j-1}^0
\right)
\geq
\underline C_1,
\end{equation}
and
\begin{equation}
\label{eq:min-jump}
\min_{1\leq j\leq k_0-1}
\|V_{j+1}^0-V_j^0\|_2
\geq
\underline C_2.
\end{equation}

\begin{lemma}[Local geometric approximation bounds]
\label{lem:local-geometry}
Fix a true changepoint $\tau_j^0$, where
$j\in\{1,\ldots,k_0-1\}$, and assume that
\eqref{eq:min-jump} is satisfied.

First, for every
$
0<\epsilon<
\frac12
\min\left\{
\tau_j^0-\tau_{j-1}^0,\,
\tau_{j+1}^0-\tau_j^0
\right\},
$
the best $\Rbb^d$-valued affine approximation to $f^0$ on
$[\tau_j^0-\epsilon,\tau_j^0+\epsilon]$ satisfies
\begin{equation}
\label{eq:cubic-local-bound}
\begin{aligned}
\inf_{\ell\in\mathcal L_d}
\int_{\tau_j^0-\epsilon}^{\tau_j^0+\epsilon}
\|f^0(t)-\ell(t)\|_2^2\,dt
&=
\frac{\epsilon^3}{24}
\|V_{j+1}^0-V_j^0\|_2^2\geq
\frac{\underline C_2^2}{24}\epsilon^3,
\end{aligned}
\end{equation}
where
$
\mathcal L_d
:=
\left\{
\ell:[0,1]\to\Rbb^d:
\ell(t)=a+vt,\;
a,v\in\Rbb^d
\right\}
$
is the class of $\Rbb^d$-valued affine functions.

Second, choose fixed points satisfying
$
\frac{\tau_{j-1}^0+\tau_j^0}{2}
<t_1<\tau_j^0<t_2<
\frac{\tau_j^0+\tau_{j+1}^0}{2}.
$
For $\tau\in(t_1,t_2)$, let
$\mathcal G_2(\tau;t_1,t_2)$ denote the class of continuous
$\Rbb^d$-valued piecewise-linear functions on $[t_1,t_2]$ with a
possible changepoint at $\tau$, namely
$$
\mathcal G_2(\tau;t_1,t_2)
:=
\left\{
g:[t_1,t_2]\to\Rbb^d:
g(t)=
\begin{cases}
a+v_1(t-t_1),
& t\in[t_1,\tau],\\[2mm]
a+v_1(\tau-t_1)+v_2(t-\tau),
& t\in[\tau,t_2],
\end{cases}
\quad
a,v_1,v_2\in\Rbb^d
\right\}.
$$
Then there exist constants $c_j>0$ and $\delta_j>0$, independent of
$\tau$, such that, whenever
$|\tau-\tau_j^0|\leq\delta_j$,
\begin{equation}
\label{eq:quadratic-local-bound}
\inf_{g\in\mathcal G_2(\tau;t_1,t_2)}
\int_{t_1}^{t_2}
\|f^0(t)-g(t)\|_2^2\,dt
\geq
c_j
|\tau-\tau_j^0|^2
\|V_{j+1}^0-V_j^0\|_2^2.
\end{equation}
Consequently, under \eqref{eq:min-jump},
\begin{equation}
\label{eq:quadratic-local-bound-min-jump}
\inf_{g\in\mathcal G_2(\tau;t_1,t_2)}
\int_{t_1}^{t_2}
\|f^0(t)-g(t)\|_2^2\,dt
\geq
c_j\underline C_2^2
|\tau-\tau_j^0|^2.
\end{equation}
\end{lemma}

\begin{theorem}[Consistency of the penalized CPL estimator]
\label{thm:consistency}
Suppose that the conditions of Lemma~\ref{lem:rates} hold, that
\eqref{eq:min-spacing}--\eqref{eq:min-jump} are satisfied, that
$\overline{k}\ge k_0$, and that $\gamma>1$. Then the penalized
maximum-likelihood estimator defined in \eqref{eq:pen-MLE} satisfies
\begin{equation}\label{eq:consistency}
  \Pbb\!\left(
    \widehat{k}_n=k_0
    \ \text{and}\
    \max_{1\le j\le k_0-1}
    \left|
      \widehat{\tau}_{n,j}-\tau_j^0
    \right|
    \le
    C\left(\frac{\log n}{n}\right)^{1/2}
  \right)
  \longrightarrow 1
\end{equation}
as $n\to\infty$, where $C>0$ does not depend on $n$.
\end{theorem}

\begin{remark}
The condition $\gamma>1$ ensures that the additional penalty for an
overfitted model dominates the stochastic likelihood gain obtained by
adding spurious segments. This is analogous to the role of strengthened
logarithmic penalties in
\citet{fryzlewicz2014segmentation,Piotr2014WBS,
BaranowskiFryzlewicz2019}.
We do not claim that $\gamma>1$ is necessary for every formulation of
the change-in-velocity problem; it is the sufficient condition under
which the present penalized maximum-likelihood consistency result is
established. 
\end{remark}

\begin{remark}
The compactness assumptions on $\mathcal{V}$ and $\mathcal{A}$ are standard
regularity conditions for the consistency analysis. In the
intracellular-transport setting motivating CPLASS, boundedness of
$\mathcal{V}$ is also scientifically natural because particle speeds are
bounded in practice. The companion methods paper
\citet{Do2025-vv} incorporates this information through a soft
speed-control penalty, whereas the present analysis studies the corresponding
bounded-velocity parameterization. Thus, the theorem concerns the global
sSIC-penalized estimator over a compact velocity space rather than the
finite-iteration stochastic-search output with the soft speed penalty.
\end{remark}

\paragraph{Proof strategy.}

Lemma~\ref{lem:rates} is proved in Appendix~\ref{sec:S-proof-rates} using
empirical-process arguments. For each fixed $k$, the Hellinger
bracketing entropy of the CPL model class satisfies
$$H_B(\epsilon)
\leq
C_B
\left\{
(k-1)\log n
+
\bigl[d(k+1)+1\bigr]\log(1/\epsilon)
\right\}.$$

The first term accounts for the discrete choices of the $k-1$
changepoint locations on the observation grid, while the second arises
from covering the continuous initial-position, velocity, and variance
parameters.

Applying Theorem~8.14 of \citet{Vandegeer} gives
$
\bar h_n^2\!\left(
p_{\widehat f_n^{(k)},\widehat\sigma_{n,k}^2},
p_{f^0,\sigma_0^2}
\right)
=
O_{\Pbb}\!\left(
\dfrac{\rho_k\log n}{n}
\right).
$
The Gaussian Hellinger equivalence in
Lemma~\ref{lem:gaussian-hellinger} then yields
\eqref{eq:fn-rate}--\eqref{eq:sigma-rate}, while
Theorem~8.13 of \citet{Vandegeer} gives the likelihood bound
\eqref{eq:llh-rate}. The discrete and continuous entropy contributions
are summarized by
$
\rho_k=d(k+1)+k.
$

The proof of Theorem~\ref{thm:consistency} has three parts. First, for
every fixed $k>k_0$,
$$
2\left\{
\overline{\mathcal L}_n
\!\left(
\widehat f_n^{(k)},\widehat\sigma_{n,k}^2
\right)
-
\overline{\mathcal L}_n
\!\left(
\widehat f_n^{(k_0)},\widehat\sigma_{n,k_0}^2
\right)
\right\}
=
O_{\Pbb}\!\left(
\frac{\log n}{n}
\right).
$$
By contrast, the normalized penalty difference is
$
(\rho_k-\rho_{k_0})
\frac{(\log n)^\gamma}{n},
$
which dominates $\log n/n$ when $\gamma>1$. Since $\overline k$ is
fixed, this excludes all overfitted models simultaneously.

Second, any model with fewer than $k_0$ segments must omit at least one
true changepoint. By the minimum-spacing condition, one can choose
pairwise disjoint fixed neighborhoods around the true changepoints. If
an underfitted signal has no changepoint in one such neighborhood, it is
affine there. The cubic one-segment approximation bound in
Lemma~\ref{lem:local-geometry} therefore gives a fixed positive
population approximation gap:
$$
\inf_{k<k_0}
\inf_{f\in\mathcal F_k}
\int_0^1
\|f(t)-f^0(t)\|_2^2\,dt
\geq
\Delta_0
$$
for some $\Delta_0>0$. A uniform Riemann-sum approximation transfers
this gap to the empirical norm, while the uniform stochastic
cross-term bound established in Appendix~\ref{sec:S-proof-consistency} yields
$
\inf_{k<k_0}
\frac{1}{dn}
\sum_{i=1}^n
\left\|
Y_i-\widehat f_n^{(k)}(t_i)
\right\|_2^2
\geq
\sigma_0^2+c_0
$
with probability tending to one, for some $c_0>0$. After profiling over
the variance parameter, this produces a fixed positive likelihood gap
between every underfitted model and the true model. Because the
corresponding penalty difference is only of order
$(\log n)^\gamma/n$, underfitting is excluded.

Finally, conditional on selecting $k_0$ segments, changepoint
localization follows from the two geometric bounds in
Lemma~\ref{lem:local-geometry}. Lemma~\ref{lem:rates} and the cubic
one-segment approximation bound \eqref{eq:cubic-local-bound} first give
$
\max_{1\leq j\leq k_0-1}
\left|
\widehat\tau_{n,j}-\tau_j^0
\right|
=
O_{\Pbb}\!\left\{
\left(\dfrac{\log n}{n}\right)^{1/3}
\right\}.
$
This preliminary rate establishes consistency and, with probability
tending to one, a one-to-one correspondence between the estimated and
true changepoints. On a fixed neighborhood $[t_1,t_2]$ of each true
changepoint, the true and fitted signals then each have exactly two
segments. Applying the quadratic two-segment approximation bound
\eqref{eq:quadratic-local-bound} gives
$
\max_{1\leq j\leq k_0-1}
\left|
\widehat\tau_{n,j}-\tau_j^0
\right|
=
O_{\Pbb}\!\left\{
\left(\frac{\log n}{n}\right)^{1/2}
\right\}.
$

Full proofs are provided in the appendices.
\section*{Organization of the appendices}

The appendices contain the complete proofs of
Lemmas~\ref{lem:gaussian-hellinger}--\ref{lem:local-geometry} and
Theorem~\ref{thm:consistency}, including the Hellinger bracketing-entropy
bounds, empirical-process concentration arguments, and the complete
two-stage geometric derivation of the changepoint-location rate.

\appendix
\setcounter{theorem}{0}\setcounter{equation}{0}\setcounter{remark}{0}\setcounter{definition}{0}
\renewcommand{\thetheorem}{S.\arabic{theorem}}
\renewcommand{\theequation}{S.\arabic{equation}}
\renewcommand{\theremark}{S.\arabic{remark}}
\renewcommand{\thedefinition}{S.\arabic{definition}}
\renewcommand{\theHtheorem}{S.\arabic{theorem}}
\renewcommand{\theHequation}{S.\arabic{equation}}
\renewcommand{\theHremark}{S.\arabic{remark}}
\renewcommand{\theHdefinition}{S.\arabic{definition}}

\section{Notation and model recap}\label{sec:S-model}

We recall the setup from the main text for convenience.
Observations $(Y_i)_{i=1}^n\subset\Rbb^d$ at times $t_i=i/n$ satisfy
\begin{equation}\label{eq:S-true-model}
  Y_i \overset{\mathrm{ind.}}{\sim}
  \mathcal{N}_d(f^0(t_i),\,\sigma_0^2 I_d),
\end{equation}
where $f^0=f_{\tbf^0,\Vbf^0,\av^0}\in\mathcal{F}_{k_0}$ is the true
continuous piecewise-linear (CPL) signal with $k_0$ segments,
changepoints $\tbf^0=(\tau_1^0,\dots,\tau_{k_0-1}^0)$,
velocities $\Vbf^0=(V_1^0,\dots,V_{k_0}^0)\in\mathcal{V}^{k_0}$,
and initial position $\av^0\in\mathcal{A}$.
The sets $\mathcal{V},\mathcal{A}\subset\Rbb^d$ are compact and
$\sigma^2\in\Omega=[\underline{\sigma}^2,\overline{\sigma}^2]\subset(0,\infty)$.
The CPL signal class $\mathcal{F}_k$ and the penalized MLE
$(\widehat{f}_n,\widehat{\sigma}_n^2,\widehat{k}_n)$ are defined
in Section~2 of the main text.

We write:
\begin{itemize}\setlength\itemsep{2pt}
  \item $p_{\tbf,\Vbf,\av,\sigma^2}^{(n)}
    :=\bigotimes_{i=1}^n\mathcal N_d\!\left(\,\cdot\mid
    f_{\tbf,\Vbf,\av}(t_i),\sigma^2I_d\right)$,
    \quad $p_0^{(n)}:=\bigotimes_{i=1}^n\mathcal N_d\!\left(\,\cdot\mid
    f^0(t_i),\sigma_0^2I_d\right)$.
  \item $\bar h_n^2(\mathbf p,\mathbf q)
:=
n^{-1}\sum_{i=1}^n h^2(p_i,q_i)$ for the
    average squared Hellinger process distance
    \citep{Vandegeer,gine2021mathematical}.
  \item $\widehat{f}_n^{(k)}$ denotes
    $f_{\widehat{\tbf}^{n,k},\widehat{\Vbf}^{n,k},\widehat{\av}^{n,k}}$,
    and $\|f-g\|_n^2=n^{-1}\sum_i\|f(t_i)-g(t_i)\|_2^2$.
  \item $\overline{\mathcal{L}}_n(f,\sigma^2)=n^{-1}\sum_i
    \log\mathcal{N}_d(Y_i\mid f(t_i),\sigma^2I_d)$.
\end{itemize}

For a fixed candidate model size $k$, let
$\widehat\theta=(\widehat{\tbf}^{n,k},\widehat{\Vbf}^{n,k},
\widehat{\av}^{n,k},\widehat\sigma_{n,k}^2)$ denote the MLE over the
corresponding parameter space. Thus,
$p_{\widehat\theta}^{(n)}=\bigotimes_{i=1}^n
\mathcal N_d(\,\cdot\mid\widehat f_n^{(k)}(t_i),
\widehat\sigma_{n,k}^2I_d)$.

\paragraph{Asymptotic notation.}
For sequences $(a_n),(b_n)$: $a_n\lesssim b_n$ (or $a_n=O(b_n)$) if
$a_n\le Cb_n$ for a constant $C$ not depending on $n$;
$a_n\asymp b_n$ if $a_n\lesssim b_n$ and $b_n\lesssim a_n$;
$a_n=o(b_n)$ if $a_n/b_n\to0$.
For densities $p,q$:
$h^2(p,q)=\frac12\int(\sqrt p-\sqrt q)^2$,
$V(p,q)=\frac12\int|p-q|$, and
$\mathrm{KL}(p\|q)=\int p\log(p/q)$.
These satisfy $h^2(p,q)\le V(p,q)\le\sqrt2\,h(p,q)$ and
$V(p,q)\le\{\mathrm{KL}(p\|q)/2\}^{1/2}$.

The penalized estimator can equivalently be obtained by first computing
\begin{equation}\label{eq:S-each-MLE}
  (\widehat f_n^{(k)},\widehat\sigma_{n,k}^2)
  \in\argmax_{f\in\mathcal F_k,\,\sigma^2\in\Omega}
  \sum_{i=1}^n\log\mathcal N_d(Y_i\mid f(t_i),\sigma^2I_d),
\end{equation}
for each $1\le k\le\overline k$, and then selecting
\begin{equation}\label{eq:S-sSIC}
  \widehat k_n\in\argmax_{1\le k\le\overline k}
  \left\{2\overline{\mathcal L}_n(\widehat f_n^{(k)},
  \widehat\sigma_{n,k}^2)-\rho_k\frac{(\log n)^\gamma}{n}\right\},
  \qquad \rho_k=d(k+1)+k.
\end{equation}

\begin{remark}[Off-grid true changepoints]
\label{rem:S-off-grid}

Suppose that a true changepoint $\tau_j^0$ does not lie on the
observation grid, and let $\tau_{j,n}^0$ be a nearest grid point, so that
$
|\tau_{j,n}^0-\tau_j^0|
\leq
\dfrac{1}{n}.
$
Let $f_n^0$ be the signal obtained from $f^0$ by replacing each
$\tau_j^0$ with $\tau_{j,n}^0$, while retaining the same initial
position and velocities. By the hinge representation,
$$
f_n^0(t)-f^0(t)
=
\sum_{j=1}^{k_0-1}
(V_{j+1}^0-V_j^0)
\left\{
(t-\tau_{j,n}^0)_+
-
(t-\tau_j^0)_+
\right\}.
$$
Because the positive-part map is $1$-Lipschitz,
$$
\begin{aligned}
\sup_{t\in[0,1]}
\|f_n^0(t)-f^0(t)\|_2
&\leq
\sum_{j=1}^{k_0-1}
\|V_{j+1}^0-V_j^0\|_2
|\tau_{j,n}^0-\tau_j^0|\\
&\leq
\frac{1}{n}
\sum_{j=1}^{k_0-1}
\|V_{j+1}^0-V_j^0\|_2
\leq
\frac{C}{n},
\end{aligned}
$$
where $C<\infty$ does not depend on $n$. Consequently,
$$
\|f_n^0-f^0\|_n^2
\leq
\sup_{t\in[0,1]}
\|f_n^0(t)-f^0(t)\|_2^2
\leq
\frac{C^2}{n^2}.
$$
Thus, nearest-grid discretization contributes
$O(n^{-2})$ empirical squared error, which is negligible relative to
the $O(\log n/n)$ estimation rate.
\end{remark}
\section{Empirical process tools}\label{sec:S-EPT}

 Given a parameter space
$
    \Theta
    \subseteq
    \mathcal T^{k-1}\times\mathcal V^k\times\mathcal A\times\Sigma
$
and product densities indexed by
$
    \theta=(\tbf,\Vbf,\av,\sigma^2)\in\Theta,
$
we first define the Hellinger covering number and Hellinger bracketing
entropy.

\begin{definition}\label{def:covering-number}
    For $\delta > 0$, let $N_B(\delta, \Theta)$ be the smallest integer $N$ such that there exists a collection of non-negative functions $\{\mathbf{p}_j^L, \mathbf{p}_j^U \}_{j=1}^{N}$ with $\mathbf{p}_j^{L} = (p_{j1}^{L}, \dots, p_{jn}^{L})$ and  $\mathbf{p}_j^{U} = (p_{j1}^{U}, \dots, p_{jn}^{U})$ such that for every $\theta \in \Theta$, there is a $j$ such that
    \begin{itemize}
        \item[(i)] $\bar h_n\left(
\dfrac{\mathbf p_j^L+\mathbf p_0^{(n)}}{2},
\dfrac{\mathbf p_j^U+\mathbf p_0^{(n)}}{2}
\right)
\le \delta$ and
        \item[(ii)] $p_{ji}^{L}(Y_i) \leq \mathcal{N}_d(Y_i|f_{\tbf,\Vbf,\av}(t_i),\sigma^2I_d)\leq p_{ji}^{U}(Y_i)$
        for all $Y_i\in \Rbb^d$ and $i\in [n]$.
    \end{itemize}
    Then $N_B(\delta, \Theta)$ and $H_B(\delta, \Theta)= \log N_B(\delta, \Theta)$ are called the Hellinger covering number and entropy number with bracketing, respectively. When $\Theta = \mathcal{T}^{k-1}\times \mathcal{V}^{k}\times \mathcal{A}\times \Sigma$, we write those numbers as $N_B(\delta)$ and $H_B{\delta}$ for short.
\end{definition}

The entropy integral is
$J_B(\delta)=\int_{\delta^2/c_0}^{\delta}H_B^{1/2}(u,\Theta)\,du\vee\delta$
for a universal constant $c_0$.

We state the two main theorems from \citet{Vandegeer}, with notation
adapted to our model.

\begin{theorem}[{\citet[Thm.~8.14]{Vandegeer}}]\label{thm:S-vdG-8-14}
Suppose $\Psi(\delta)\ge J_B(\delta)$ with $\Psi(\delta)/\delta^2$
non-increasing. If $\sqrt n\delta_n^2\ge c\Psi(\delta_n)$ for universal
$c>0$, then for all $\delta\ge\delta_n$,
\begin{equation}\label{eq:S-vdG-8-14}
  \Pbb_0\!\left(\overline h_n(p_{\widehat\theta}^{(n)},p_0^{(n)})
  \ge\delta\right)\le c\exp\!\left(-\frac{n\delta^2}{c^2}\right).
\end{equation}
\end{theorem}

For $\theta\in\Theta$, define $\bar p_{\theta,i}=(p_{\theta,i}+p_{0,i})/2$
and $\bar p_\theta^{(n)}=\bigotimes_{i=1}^n\bar p_{\theta,i}$.

\begin{theorem}[{\citet[Thm.~8.13]{Vandegeer}}]\label{thm:S-vdG-8-13}
Let $R,D,A_1,b>0$ and $\Theta$ satisfy
$\overline h_n(\bar p_\theta^{(n)},p_0^{(n)})\le R$ for all
$\theta\in\Theta$;
$b\le A_1\sqrt nR^2\wedge8\sqrt nR$; and
$b\ge\sqrt{D^2(A_1+1)}\bigl(\int_{b/(2^6\sqrt n)}^R
H_B^{1/2}(u/\sqrt2,\Theta)\,du\vee R\bigr)$.
With
\begin{equation*}
 Z_{i,\theta}=\tfrac12\log\frac{p_{\theta,i}(Y_i)+p_{0,i}(Y_i)}
 {2p_{0,i}(Y_i)},
\end{equation*}
$\bar Z_\theta=n^{-1}\sum_iZ_{i,\theta}$ and
$\bar A_\theta=n^{-1}\sum_i\mathbb E_0Z_{i,\theta}$,
\begin{equation}\label{eq:S-vdG-8-13}
 \Pbb_0\!\left(\sup_{\theta\in\Theta}\sqrt n
 |\bar Z_\theta-\bar A_\theta|\ge b\right)
 \le D\exp\!\left[-\frac{b^2}{D^2(A_1+1)R^2}\right].
\end{equation}
\end{theorem}

\section{Proof of Lemma~1: Gaussian Hellinger equivalence}
\label{sec:S-proof-hellinger}

\begin{restatedlemma}[Lemma~1 of the main text and an auxiliary bound]
\label{lem:S-gaussian}
    Suppose that $\mathcal M\subset\Rbb^d$ is compact and that
$\Omega\subset[\underline c,\overline c]$ with
$0<\underline c<\overline c<\infty$.
For all $\mu,\widetilde\mu\in\mathcal M$ and
$\sigma^2,\widetilde\sigma^2\in\Omega$, we have
    \begin{equation}
        h^2(\mathcal{N}_d( \mu, \sigma^2 I_d), \mathcal{N}_d( \tilde{\mu}, \tilde{\sigma}^2 I_d))\asymp \|\mu - \tilde{\mu}\|^2 + (\sigma^2 - \tilde{\sigma}^2)^2,\label{eq:hellinger_and_parameter}
    \end{equation}
    and
    \begin{equation}
        \sup_{y\in \Rbb^{d}} \left|\mathcal{N}_d(y\mid\mu, \sigma^2 I_d) - \mathcal{N}_d(y\mid \tilde{\mu}, \tilde{\sigma}^2 I_d)\right|\lesssim \|\mu - \tilde{\mu}\| + |\sigma^2 - \tilde{\sigma}^2|.\label{eq:infinity_norm_parameter}
    \end{equation}
\end{restatedlemma}
\begin{proof}

\noindent\textbf{1. Proof of \eqref{eq:hellinger_and_parameter}:}

Recall the squared Hellinger distance between two isotropic Gaussian densities:
\begin{align*}
h^2\left(\mathcal{N}_d(\mu,\sigma^2I_d),
\mathcal{N}_d(\tilde{\mu},\tilde{\sigma}^2I_d)\right)=1-\left(\dfrac{2\sqrt{\sigma^2\tilde{\sigma}^2}}
{\sigma^2+\tilde{\sigma}^2}\right)^{d/2}
\exp\left\{-\dfrac{\|\mu-\tilde{\mu}\|^2}
{4(\sigma^2+\tilde{\sigma}^2)}\right\}.
\end{align*}
Write
\[
  A=
  \left(
    \frac{2\sqrt{\sigma^2\tilde{\sigma}^2}}
    {\sigma^2+\tilde{\sigma}^2}
  \right)^{d/2},
  \qquad
  B=
  \frac{\|\mu-\tilde{\mu}\|_2^2}
  {4(\sigma^2+\tilde{\sigma}^2)}.
\]
Then
\[
  h^2=1-Ae^{-B}=(1-A)+A(1-e^{-B}).
\]
Since $\Omega$ is compact and bounded away from zero,
$A$ is uniformly bounded away from zero. Moreover,
\[
  B\asymp\|\mu-\tilde{\mu}\|_2^2,
  \qquad
  1-e^{-B}\asymp B,
\]
where the second equivalence follows because the mean space is compact.
Thus,
\[
  A(1-e^{-B})
  \asymp
  \|\mu-\tilde{\mu}\|_2^2.
\]

Also,
\[
  A=
  \left\{
    1-
    \frac{(\sigma^2-\tilde{\sigma}^2)^2}
    {(\sigma^2+\tilde{\sigma}^2)^2}
  \right\}^{d/4}.
\]
Since $\Omega$ is compact and bounded away from zero,
\[
  1-A
  \asymp
  \frac{(\sigma^2-\tilde{\sigma}^2)^2}
  {(\sigma^2+\tilde{\sigma}^2)^2}
  \asymp
  (\sigma^2-\tilde{\sigma}^2)^2.
\]
Therefore,
\[
  h^2\bigl(
    \mathcal N_d(\mu,\sigma^2I_d),
    \mathcal N_d(\tilde\mu,\tilde\sigma^2I_d)
  \bigr)
  \asymp
  \|\mu-\tilde\mu\|_2^2
  +
  (\sigma^2-\tilde\sigma^2)^2.
\]

\noindent\textbf{2. Proof of \eqref{eq:infinity_norm_parameter}:}
We have that 
\begin{align*}
    \mathcal{N}_d(y\mid\mu, \sigma^2 I_d) = \dfrac{1}{(2\pi\sigma^2)^{d/2}}\exp\left(-\dfrac{\norm{y-\mu}^2}{2\sigma^2}\right)
\end{align*}

The derivative with respect to $\mu$ and $\sigma^2$ are then
\begin{align}
    \dfrac{\partial}{\partial \mu} \mathcal{N}_d(y\mid\mu,\sigma^2I_d) &= \mathcal{N}_d(y\mid\mu,\sigma^2I_d)\cdot \dfrac{y-\mu}{\sigma^2},\\
    \dfrac{\partial}{\partial \sigma^2} \mathcal{N}_d(y\mid\mu,\sigma^2I_d) &= \mathcal{N}_d(y\mid\mu,\sigma^2I_d)\cdot \left(-\dfrac{d}{2\sigma^2}+\dfrac{\|y-\mu\|^2}{2\sigma^4}\right).
\end{align}

By Cauchy-Schwarz inequality and triangle inequality, we have that:
\begin{align}
    \left|\dfrac{\partial}{\partial \mu} \mathcal{N}_d(y\mid\mu,\sigma^2I_d)\right| &\le \mathcal{N}_d(y\mid\mu,\sigma^2I_d)\cdot \dfrac{\|y-\mu\|}{\sigma^2},\\
    \left|\dfrac{\partial}{\partial \sigma^2} \mathcal{N}_d(y\mid\mu,\sigma^2I_d)\right| &\le \mathcal{N}_d(y\mid\mu,\sigma^2I_d)\cdot \left(\dfrac{d}{2\sigma^2}+\dfrac{\|y-\mu\|^2}{2\sigma^4}\right).
\end{align}
From the Mean Value Theorem, the difference $\left|\mathcal{N}_d(y\mid\mu, \sigma^2 I_d) - \mathcal{N}_d(y\mid \tilde{\mu}, \tilde{\sigma}^2 I_d)\right|$ can be expressed in terms of the gradients with respect to $\mu$ and $\sigma^2$:
\begin{align}
    \sup_{y \in \mathbb{R}^d}\left|\mathcal{N}_d(y\mid\mu, \sigma^2 I_d) - \mathcal{N}_d(y\mid \tilde{\mu}, \tilde{\sigma}^2 I_d)\right| \le L_{\mu} \|\mu - \tilde{\mu}\| + L_{\sigma^2}|\sigma^2 - \tilde{\sigma}^2|,
\end{align}
where 
{\fontsize{10}{10}\selectfont 
\begin{align*}
    L_{\mu} = \sup_{y \in \mathbb{R}^d}\dfrac{\|y-\mu\|}{\sigma^2}\mathcal{N}_d(y\mid\mu,\sigma^2I_d) = \sup_{y \in \mathbb{R}^d}\underbrace{\dfrac{1}{(2\pi)^{d/2}\sigma^{2(d/2+1)}} }_{\text{is bounded since } \sigma^2 \in [\underline{\sigma}^2,\overline{\sigma}^2]}\underbrace{\norm{y-\mu}\exp\left(-\dfrac{\norm{y-\mu}^2}{2\sigma^2}\right)}_{\text{is bounded}} \le C',
\end{align*}}
and
{\fontsize{10}{10}\selectfont 
 \begin{align*}
     &L_{\sigma^2} = \sup_{y\in \mathbb{R}^d}\left(\dfrac{d}{2\sigma^2}+\dfrac{\|y-\mu\|^2}{2\sigma^4}\right)\mathcal{N}_d(y\mid\mu,\sigma^2I_d)\\ &= \sup_{y\in \mathbb{R}^d}\left(\dfrac{d}{2\sigma^2}+\dfrac{\|y-\mu\|^2}{2\sigma^4}\right)  \dfrac{1}{(2\pi\sigma^2)^{d/2}}\exp\left(-\dfrac{\norm{y-\mu}^2}{2\sigma^2}\right)\\
     & = \sup_{y\in \mathbb{R}^d}\underbrace{\dfrac{d}{2(2\pi)^{d/2}\sigma^{2(d/2+1)}}}_{\text{is bounded since } \sigma^2 \in [\underline{\sigma}^2, \overline{\sigma}^2]}\underbrace{\exp\left(-\dfrac{\|y-\mu\|^2}{2\sigma^2}\right)}_{\text{is bounded }}+ \underbrace{\dfrac{1}{(2\pi)^{d/2}\sigma^{2(d/2+2)}}}_{\text{is bounded since } \sigma^2 \in [\underline{\sigma}^2, \overline{\sigma}^2]}\underbrace{\norm{y-\mu}^2\exp\left(-
    \dfrac{\|y-\mu\|^2}{2\sigma^2}\right)}_{\text{is bounded}} \le C''
 \end{align*}} are Lipschitz constants.

We then obtain 
$$\sup_{y\in \Rbb^{d}} \left|\mathcal{N}_d(y\mid\mu, \sigma^2 I_d) - \mathcal{N}_d(y\mid \tilde{\mu}, \tilde{\sigma}^2 I_d)\right|\lesssim \|\mu - \tilde{\mu}\| + |\sigma^2 - \tilde{\sigma}^2|.$$

\end{proof}
\section{Proof of Lemma~2: Fixed-model convergence rates}
\label{sec:S-proof-rates}

To prepare for the next result, we define the average empirical $L_2$ distance between two functions $f,g: [0,1] \to \mathbb{R}^d $ as follows
\begin{align}
    \|f-g\|_n = \left(\dfrac{1}{n} \sum_{i=1}^n \|f(i/n)-g(i/n)\|^2\right)^{1/2}.
\end{align}

\begin{restatedlemma}[Lemma~2 of the main text]
\label{lem:S-rates}
Under the conditions of Lemma~2 in the main text, for each fixed $k \geq k_0$, there exist universal constants $c_1, c_2 >0$ such that with at least probability $1- c_1n^{-c_2}$  we have 
\begin{equation}\label{eq:conv-function-each-MLE}
\|\widehat{f}_n^{(k)} - f^0\|^{2}_{n} \leq C\left(\dfrac{\rho_k\log n}{n}\right),\quad\quad |\widehat{\sigma}^2_{n, k} - \sigma_0^2| \leq C\left(\dfrac{\rho_k\log n}{n}\right)^{1/2},
\end{equation}
and
\begin{equation}
\label{eq:conv-llh-each-MLE}
0
\leq
L_n\left(
\widehat f_n^{(k)},
\widehat\sigma_{n,k}^2
\right)
-
L_n\left(
f^0,
\sigma_0^2
\right)
\leq
C\rho_k\frac{\log n}{n}.
\end{equation}
where $\rho_k=d(k+1)+k$ is the number of parameters of the model, and $C$ may depend on $d$, $k$, $\mathcal{V}$, $\mathcal{A}$, and $\Omega$, but not on $n$.
\end{restatedlemma}
\begin{proof}
The proof is divided into a few small steps.
\paragraph{Step 1. Bound the covering number of the space of changepoint models} To apply Theorem~\ref{thm:S-vdG-8-14} to provide the convergence of parameters and likelihood, we need to bound the covering number with bracketing for the space of changepoint models with $k$ segments. We generalize the technique in~\cite{van_de_vaart} for our model.

\paragraph{Step 1.1. Covering the space of changepoint models with fixed
changepoints under the $\ell_{\infty}$ norm.}
Suppose that the changepoint vector
$\tbf=(\tau_1,\dots,\tau_{k-1})\subset(0,1)$ is fixed.

Given $\varepsilon>0$, compactness of
$\mathcal A\subset\Rbb^d$ implies that there exists an
$\varepsilon$-net
$\{\av_i\}_{i=1}^{N_1}\subset\mathcal A$ under the
$\ell_\infty$ norm, with
$
N_1\lesssim
\left(\frac{1}{\varepsilon}\right)^d.
$
Thus, for every $\av\in\mathcal A$, there exists
$\av_i$ satisfying
$
\|\av-\av_i\|_\infty<\varepsilon.
$
Similarly, there exist $\varepsilon$-nets
$\{V_i\}_{i=1}^{N_2}\subset\mathcal V$ and
$\{\sigma_i^2\}_{i=1}^{N_3}\subset\Omega$ satisfying
$
N_2\lesssim
\left(\frac{1}{\varepsilon}\right)^d,
N_3\lesssim
\frac{1}{\varepsilon}.
$

Consider the product net
\begin{align*}
B
:=
\bigl\{
&(\av_{i_0},V_{i_1},\dots,V_{i_k},\sigma_{i_{k+1}}^2):
i_0\in[N_1],i_1,\dots,i_k\in[N_2],\;
i_{k+1}\in[N_3]
\bigr\}
\subset
\mathcal A\times\mathcal V^k\times\Omega.
\end{align*}

Its cardinality satisfies
$
|B|
\lesssim
\left(\frac{1}{\varepsilon}\right)^{d(k+1)+1}.
$
Moreover, for every
$(\av,\Vbf,\sigma^2)\in
\mathcal A\times\mathcal V^k\times\Omega$,
there exists
$(\widetilde{\av},\widetilde{\Vbf},\widetilde{\sigma}^2)\in B$
such that
$
\|\av-\widetilde{\av}\|_\infty<\varepsilon,
\|V_j-\widetilde V_j\|_\infty<\varepsilon,
j=1,\dots,k,
$
and
$
|\sigma^2-\widetilde{\sigma}^2|<\varepsilon.
$

Using the hinge representation of a continuous piecewise-linear signal,
for every $t\in[0,1]$,
$$
f_{\tbf,\Vbf,\av}(t)
=
\av+V_1t
+
\sum_{j=1}^{k-1}
(V_{j+1}-V_j)(t-\tau_j)_+.
$$
Because the changepoint vector $\tbf$ is fixed for both functions,
$$
\begin{aligned}
&\left\|
f_{\tbf,\Vbf,\av}(t)
-
f_{\tbf,\widetilde{\Vbf},\widetilde{\av}}(t)
\right\|_\infty\leq
\|\av-\widetilde{\av}\|_\infty
+
\|V_1-\widetilde V_1\|_\infty t+
\sum_{j=1}^{k-1}
\left\|
(V_{j+1}-V_j)
-
(\widetilde V_{j+1}-\widetilde V_j)
\right\|_\infty
(t-\tau_j)_+.
\end{aligned}
$$
Since $t\leq1$, $(t-\tau_j)_+\leq1$, and
$
\left\|
(V_{j+1}-V_j)
-
(\widetilde V_{j+1}-\widetilde V_j)
\right\|_\infty
\leq
2\varepsilon,
$
we obtain
$$
\begin{aligned}
\left\|
f_{\tbf,\Vbf,\av}(t)
-
f_{\tbf,\widetilde{\Vbf},\widetilde{\av}}(t)
\right\|_\infty
&\leq
\varepsilon+\varepsilon+2(k-1)\varepsilon=
2k\varepsilon.
\end{aligned}
$$
Therefore,
$
\sup_{t\in[0,1]}
\left\|
f_{\tbf,\Vbf,\av}(t)
-
f_{\tbf,\widetilde{\Vbf},\widetilde{\av}}(t)
\right\|_\infty
\leq
2k\varepsilon.
$

Combining this bound with
Lemma~\ref{lem:S-gaussian}, we obtain
\begin{equation}
\label{eq:S-density-net-bound}
\begin{aligned}
\sup_{y\in\Rbb^d}
\Bigl|
&\mathcal N_d
\bigl(
y\mid
f_{\tbf,\Vbf,\av}(t_i),
\sigma^2I_d
\bigr)-
\mathcal N_d
\bigl(
y\mid
f_{\tbf,\widetilde{\Vbf},\widetilde{\av}}(t_i),
\widetilde{\sigma}^2I_d
\bigr)
\Bigr|
\lesssim
k\varepsilon,
\qquad i\in[n].
\end{aligned}
\end{equation}
\paragraph{Step 1.2. Covering the space of changepoint models with fixed changes under Hellinger distance (with bracketing).} For every $\delta > 0$, from the previous step, we have a collection of product normal densities $\{\mathbf{p}_j\}_{j=1}^{N}$ with $\mathbf{p}_j = (p_{j1}, \dots, p_{jn})$ on $\mathcal{Y}^{n}$ and $N\asymp (k / \delta)^{d(k+1)+1}$ such that for every tuple $(\av, \Vbf, \sigma^2)\in \mathcal{A} \times \mathcal{V}^{k}\times \Omega$, there exists a $\mathbf{p}_j$ satisfying
\begin{equation}\label{eq:ell-inf-covering}
    \sup_{y\in\Rbb^d}
    \left|
    \mathcal{N}_d\!\left(
        y\mid f_{\tbf,\Vbf,\av}(t_i),\sigma^2I_d
    \right)
    -
    p_{ji}(y)
    \right|
    \leq \delta,
    \qquad \forall i\in[n].
\end{equation}
Because $\mathcal A$ and $\mathcal V$ are compact and $k$ is fixed, there
exists a compact set $\mathcal M\subset\Rbb^d$ such that
\(
    f_{\tbf,\Vbf,\av}(t_i)\in\mathcal M
\)
for every admissible $(\tbf,\Vbf,\av)$, every $i\in[n]$, and every $n$.
Moreover,
\(
    \Omega=
    [\underline{\sigma}^2,\overline{\sigma}^2]
    \subset(0,\infty).
\)
Therefore, the entire Gaussian family
\(
    \left\{
    \mathcal N_d(\,\cdot\mid\mu,\sigma^2I_d):
    \mu\in\mathcal M,\ \sigma^2\in\Omega
    \right\}
\)
admits a common integrable envelope. In particular, there exist constants
$b_1,b_2,B>0$, depending only on $\mathcal M$ and $\Omega$, such that
\begin{equation}
    H(y)=
    \begin{cases}
        b_1\exp\!\left(-b_2\norm{y}^2\right),
        & \norm{y}\ge B,\\[1mm]
        (2\pi\underline{\sigma}^2)^{-d/2},
        & \text{otherwise},
    \end{cases}
\end{equation}
satisfies
\[
    \mathcal N_d(y\mid\mu,\sigma^2I_d)\le H(y)
\]
for every $y\in\Rbb^d$, $\mu\in\mathcal M$, and
$\sigma^2\in\Omega$. Thus, $H$ dominates both the net densities
$p_{ji}$ and every candidate density in the model class.

We construct brackets
$[\mathbf p_j^L,\mathbf p_j^U]$, where
$\mathbf p_j^U=(p_{j1}^U,\ldots,p_{jn}^U)$ and
$\mathbf p_j^L=(p_{j1}^L,\ldots,p_{jn}^L)$, by setting
\begin{align*}
    p_{ji}^L(y)
    &=
    \max\{p_{ji}(y)-\delta,0\},\\
    p_{ji}^U(y)
    &=
    \min\{p_{ji}(y)+\delta,H(y)\}.
\end{align*}

For a candidate density
\[
    p_{\tbf,\Vbf,\av,\sigma^2,i}(y)
    :=
    \mathcal N_d\!\left(
    y\mid f_{\tbf,\Vbf,\av}(t_i),\sigma^2I_d
    \right),
\]
choose a net density $p_{ji}$ satisfying
\[
    \sup_{y\in\Rbb^d}
    \left|
    p_{\tbf,\Vbf,\av,\sigma^2,i}(y)-p_{ji}(y)
    \right|
    \le\delta,
    \qquad i\in[n],
\]
as guaranteed by \eqref{eq:ell-inf-covering}. Then
\[
    p_{\tbf,\Vbf,\av,\sigma^2,i}(y)
    \ge p_{ji}(y)-\delta
    \ge p_{ji}^L(y).
\]
Similarly,
\[
    p_{\tbf,\Vbf,\av,\sigma^2,i}(y)
    \le p_{ji}(y)+\delta,
\]
and, because $H$ is an envelope for the entire candidate family,
\[
    p_{\tbf,\Vbf,\av,\sigma^2,i}(y)\le H(y).
\]
Consequently,
\[
    p_{\tbf,\Vbf,\av,\sigma^2,i}(y)
    \le
    \min\{p_{ji}(y)+\delta,H(y)\}
    =
    p_{ji}^U(y).
\]
Hence,
\begin{equation}
    p_{ji}^L(y)
    \le
    \mathcal N_d\!\left(
    y\mid f_{\tbf,\Vbf,\av}(t_i),\sigma^2I_d
    \right)
    \le
    p_{ji}^U(y),
    \qquad
    \forall y\in\Rbb^d,\ i\in[n].
\end{equation}
Therefore, this collection of brackets satisfies condition~(ii) in
Definition~\ref{def:covering-number}.

We next verify condition~(i). For any $j,i$ and
$\overline B\ge B$,
\begin{align}
    \int_{\Rbb^d}
    \left(p_{ji}^U-p_{ji}^L\right)\,dy
    &\le
    \int_{\norm{y}\le\overline B}2\delta\,dy
    +
    \int_{\norm{y}\ge\overline B}H(y)\,dy
    \nonumber\\
    &\lesssim
    \delta\overline B^d
    +
    \overline B^d
    \exp\!\left(-b_2\overline B^2\right).
    \label{bracket-radi}
\end{align}

where we use spherical coordinates to have
\begin{equation*}
	 \int_{\norm{y} \leq \overline{B}}  dy = \dfrac{\pi^{d/2}}{\Gamma(d/2 + 1)} \overline{B}^d \lesssim \overline{B}^d,
\end{equation*}
and, using the change of variables $u=b_2r^2$,
\begin{align*}
\int_{\norm{y}\ge\overline B}
    \exp\!\left(-b_2\norm{y}^2\right)\,dy
&\lesssim
\int_{\overline B}^{\infty}
    r^{d-1}\exp(-b_2r^2)\,dr\\
&=
\frac{1}{2b_2^{d/2}}
\int_{b_2\overline B^2}^{\infty}
    u^{d/2-1}e^{-u}\,du\\
&\lesssim
\overline B^{d-2}
\exp\!\left(-b_2\overline B^2\right),
\end{align*}
Choose
$
    \overline B
    =
    C_0\left\{\log(1/\delta)\right\}^{1/2},
$
where $C_0>0$ is sufficiently large that $b_2C_0^2>1$.
Then, for all sufficiently small $\delta$, $\overline B\ge B$ and
\[
    \overline B^d
    \exp\!\left(-b_2\overline B^2\right)
    \lesssim
    \delta\left\{\log(1/\delta)\right\}^{d/2}.
\]
Consequently, \eqref{bracket-radi} gives
\begin{equation}\label{bracket-bound}
    \int_{\Rbb^d}
    \left(p_{ji}^U-p_{ji}^L\right)\,dy
    \lesssim
    \delta
    \left\{\log(1/\delta)\right\}^{d/2}.
\end{equation}
Moreover, let
\(
    p_i^0(y)
    :=
    \mathcal N_d\!\left(
        y\mid f^0(t_i),\sigma_0^2I_d
    \right)
\)
denote the density of $Y_i$ under the true model. Because $p^{U}_{ji} \geq p^{L}_{ji}$, we have
\begin{align*}
    h^2\left(\dfrac{p_{ji}^{U} + p_i^0}{2}, \dfrac{p_{ji}^{L} + p_i^0}{2}\right) &=\int_{\mathbb{R}^{d}} \left(\sqrt{\dfrac{p_{ji}^{U} + p_i^0}{2}} -\sqrt{\dfrac{p_{ji}^{L} + p_i^0}{2}}\right)^2 dy  \\
    & \leq \int_{\mathbb{R}^{d}} \left({\dfrac{p_{ji}^{U} + p_i^0}{2}} -{\dfrac{p_{ji}^{L} + p_i^0}{2}}\right) dy\\
    & = \dfrac{1}{2} \int_{\mathbb{R}^{d}} (p_{ji}^{U} - p_{ji}^{L}) dy\\
    & \lesssim \delta \left(\log\left(\dfrac{1}{\delta}\right)\right)^{d/2}.
\end{align*}
Therefore,
$$\overline{h}_n\left(\dfrac{\mathbf{p}_{j}^{U} + p^{(n)}_0}{2}, \dfrac{\mathbf{p}_{j}^{L} + p^{(n)}_0}{2}\right) \lesssim \delta^{1/2}(\log(1/\delta))^{d/4}.$$
Hence, there exists a constant $c>0$, independent of $\delta$, such that
\begin{align*}
H_B\!\left(
    c\delta^{1/2}
    \{\log(1/\delta)\}^{d/4}
\right)
&\leq \log N \lesssim
\{d(k+1)+1\}\log(1/\delta),
\end{align*}
Let $ \epsilon
    =
    c\delta^{1/2}
    \{\log(1/\delta)\}^{d/4}.$ Since
$
    \log(1/\epsilon)\asymp\log(1/\delta)
$ as $\delta\downarrow0$, it follows that
\begin{equation*}
    H_B(\epsilon)
    \lesssim
    \{d(k+1)+1\}\log(1/\epsilon)
\end{equation*}
for all sufficiently small $\epsilon$.
 
\paragraph{Step 1.3. Aggregate changepoints} Because there are $\binom{n}{k-1}$ ways to choose $k-1$ changepoints among $n$ data points, the covering number with bracketing of the whole model can be bounded as follows:
\begin{equation*}
    N_B(\epsilon)
     \lesssim \binom{n}{k-1} (1/\epsilon)^{d(k+1)+1} \leq n^{k-1}\left(\dfrac{1}{\epsilon}\right)^{d(k+1)+1}.
\end{equation*}
Hence, the entropy number with bracketing of the whole model can be bounded as
\begin{equation*}
    H_B(\epsilon)
     \lesssim (k-1) \log n + (d(k+1)+1)\log\left(\dfrac{1}{\epsilon}\right).
\end{equation*}

In particular, there exists $C_B>0$, depending only on $\mathcal{V}, \mathcal{A}$ and $\Omega$ such that for $n$ sufficiently large, $H_B(\epsilon) \le C_B\left[(k-1) \log n + (d(k+1)+1)\log\left(\dfrac{1}{\epsilon}\right)\right]$.

\paragraph{Step 2. Convergence rate of parameter estimation} This is a consequence of Theorem~\ref{thm:S-vdG-8-14}.

Since $\log(1/u)$ is a non-increasing function of $u$, for all $u \in [\varepsilon^2/c_0,\varepsilon]$ we have $\log (1/u) \le \log(1/(\varepsilon^2/c_0))$, and therefore
\begin{align*}
    J_B(\delta) & \leq \int_{\varepsilon^2/c_0}^{\varepsilon} \left\{C_B\left[(k-1)\log n+ (kd+d+1)\log\left(\dfrac{1}{u}\right)\right]\right\}^{1/2} du \vee \varepsilon\\
    &\leq C_B^{1/2} \varepsilon \left[(k-1)\log n+(kd+d+1)\log\dfrac{1}{ (\varepsilon^2/c_0)}\right]^{1/2}\\
    &\leq C_B'^{1/2} \varepsilon \left[(k-1)\log n+(kd+d+1)\log\dfrac{1}{ \varepsilon^2}\right]^{1/2},
\end{align*}
for all $\varepsilon$ small enough and some constant $C_B' \ge C_B$ depending only on $\mathcal{V}, \mathcal{A}$ and $\Omega$. 

Let $\Psi(\varepsilon) = C_B'^{1/2} \varepsilon \left[(k-1)\log n+(kd+d+1)\log\dfrac{1}{ \varepsilon^2}\right]^{1/2}$. For $0<\epsilon\leq1$, both $\epsilon^{-1}$ and
$
\left[
(k-1)\log n
+
(kd+d+1)\log\left(\frac1{\epsilon^2}\right)
\right]^{1/2}
$
are non-increasing functions of $\epsilon$. Therefore,
$$
\frac{\Psi(\epsilon)}{\epsilon^2}
=
\frac{C_B'^{1/2}}{\epsilon}
\left[
(k-1)\log n
+
(kd+d+1)\log\left(\frac1{\epsilon^2}\right)
\right]^{1/2}
$$
is non-increasing on $(0,1]$, as required by
Theorem~\ref{thm:S-vdG-8-14}.

Let $\varepsilon_n = \max\{ 1, 2c C_B'^{1/2}\} \left(\underbrace{(kd+d+k)}_{\rho_k}\dfrac{\log n}{n}\right)^{1/2}$ ($c>0$ is a given universal constant mentioned in Theorem~\ref{thm:S-vdG-8-14}), we have
$$c\Psi(\varepsilon_n) = cC_B'^{1/2} \varepsilon_n \left[(k-1)\log n+(kd+d+1)\log\dfrac{1}{ \varepsilon_n^2}\right]^{1/2} \leq \varepsilon_n \times (2cC_B'^{1/2}(\rho_k\log(n))^{1/2})\leq  \varepsilon_n^2 \sqrt{n}.$$
Substituting $\epsilon=\epsilon_n$ into the conclusion of
Theorem~\ref{thm:S-vdG-8-14}), we obtain
{\fontsize{10}{10}\selectfont
\begin{align*}
\mathbb{P}_0\left(\overline{h}_n\left(p^{(n)}_{\widehat{\tbf}, \widehat{\Vbf}, \widehat{\av}, \widehat{\sigma}^2}, p^{(n)}_0\right)\geq \max\{1,2cC_B^{1/2}\}\left(\dfrac{\rho_k\log n}{n}\right)^{1/2}\right) &\leq c \exp\left(-\left(\max\{ 1, 2c C_B^{1/2}\}\right)^2\rho_k\log(n)/c^2\right)\\
&\leq c_1 n^{-c_2}, (\text{since } \rho_k
\ge 1 \text{ and} (\max\{1,2cC_B^{1/2}\})^2 \ge 1  )
\end{align*}}
where $C_B$ depends on $\mathcal{V}, \mathcal{A}$, and $\Omega$ only, and $c_1=c$ and $c_2=\dfrac{1}{c^2}$ are universal constants.

By the Gaussian Hellinger equivalence in
Lemma~1, we have that $$\|\widehat{f}_n^{(k)} - f^0\|_n ^2 +|\widehat{\sigma}^2_{n, k} - \sigma_0^2|^2\asymp \overline{h}_n^2\left(p^{(n)}_{\widehat{\tbf}, \widehat{\Vbf}, \widehat{\av}, \widehat{\sigma}^2}, p^{(n)}_0\right),$$ therefore,
\begin{equation}
\|\widehat{f}_n^{(k)} - f^0\|^{2}_{n} \leq C\left(\dfrac{\rho_k\log n}{n}\right),\quad\quad |\widehat{\sigma}^2_{n, k} - \sigma_0^2| \leq C\left(\dfrac{\rho_k\log n}{n}\right)^{1/2}, \quad \text{ for } k\ge k_0
\end{equation}
where $C$ depends on $\mathcal{V}, \mathcal{A}$, and $\Omega$ only, and $c_1$ and $c_2$ are universal constants.

\paragraph{Step 3. Convergence rate of likelihood functions}

For a sequence of measurable functions
$g=(g_1,\dots,g_n)$, write
\[
  P_ng=\frac{1}{n}\sum_{i=1}^n g_i(Y_i),
  \qquad
  P_0g=\frac{1}{n}\sum_{i=1}^n
  \mathbb{E}_0\{g_i(Y_i)\},
\]
where $\mathbb{E}_0$ denotes expectation under
$Y_i\sim\mathcal{N}_d(f^0(t_i),\sigma_0^2I_d)$.

For fixed $k\ge k_0$, let
\[
  \widehat p_i(y)
  :=
  \mathcal{N}_d\!\left(
    y\mid\widehat f_n^{(k)}(t_i),
    \widehat\sigma_{n,k}^2I_d
  \right),
  \qquad
  p_i^0(y)
  :=
  \mathcal{N}_d\!\left(
    y\mid f^0(t_i),\sigma_0^2I_d
  \right),
\]
and define
\[
  \bar p_i(y)
  :=
  \frac{\widehat p_i(y)+p_i^0(y)}{2},
  \qquad
  \bar p_{\widehat\theta}^{(n)}
  :=
  \bigotimes_{i=1}^n\bar p_i.
\]
We work with the ratio $\bar p_i/p_i^0$, which is bounded below by
$1/2$, rather than $\widehat p_i/p_i^0$, which is not necessarily bounded away from zero.

By concavity of the logarithm,
\begin{align}
  \frac{1}{2}
  \log\frac{\widehat p_i(Y_i)}{p_i^0(Y_i)}
  \le
  \log\frac{\widehat p_i(Y_i)+p_i^0(Y_i)}
  {2p_i^0(Y_i)}
  =
  \log\frac{\bar p_i(Y_i)}{p_i^0(Y_i)}.
  \label{eq:Jensen_ineq}
\end{align}

Recall that we have proved
\[
  \mathbb{P}_0\left\{
    \overline h_n\left(
      p_{\widehat\theta}^{(n)},p_0^{(n)}
    \right)
    \le
    C\left(
      \frac{\rho_k\log n}{n}
    \right)^{1/2}
  \right\}
  \ge
  1-c_1n^{-c_2}.
\]
For any two densities $p$ and $p_0$,
\[
  h\left(\frac{p+p_0}{2},p_0\right)
  \le
  h(p,p_0)
  \le
  4h\left(\frac{p+p_0}{2},p_0\right).
\]
Therefore,
\[
  \mathbb{P}_0\left\{
    \overline h_n\left(
      \bar p_{\widehat\theta}^{(n)},p_0^{(n)}
    \right)
    \le
    C\left(
      \frac{\rho_k\log n}{n}
    \right)^{1/2}
  \right\}
  \ge
  1-c_1n^{-c_2}.
\]

Substitute
\[
  R
  =
  C\left(
    \frac{\rho_k\log n}{n}
  \right)^{1/2},
  \qquad
  b
  =
  C\frac{\rho_k\log n}{\sqrt n}
\]
into Theorem~\ref{thm:S-vdG-8-13}.
For sufficiently large $n$, we have
\(
  b\le \sqrt n\,R^2\wedge \sqrt n\,R,
\)
and
\begin{align*}
  \int_{b/(2^6\sqrt n)}^R
  H_B^{1/2}\left(
    \frac{u}{\sqrt2},\Theta
  \right)\,du
  &\lesssim
  R\left[
    (k-1)\log n
    +
    \{d(k+1)+1\}
    \log\left(
      \frac{2^6\sqrt n}{b}
    \right)
  \right]^{1/2}\\
  &\lesssim
  R(\rho_k\log n)^{1/2}
  \asymp b \text{ (see Step 2).}
\end{align*}
Thus, the conditions of
Theorem~\ref{thm:S-vdG-8-13} are satisfied. Applying that theorem to the
Hellinger ball
\[
  \Theta_R
  :=
  \left\{
    \theta\in\Theta:
    \overline h_n\!\left(
      \bar p_\theta^{(n)},p_0^{(n)}
    \right)
    \le R
  \right\},
\]
where
\[
  \bar p_{\theta,i}
  :=
  \frac{p_{\theta,i}+p_i^0}{2},
  \qquad
  \bar p_\theta^{(n)}
  :=
  \bigotimes_{i=1}^n\bar p_{\theta,i},
\]
gives
{\fontsize{9}{9}\selectfont
\begin{align}
  &\mathbb P_0\left\{
    \sup_{\theta\in\Theta_R}
    \sqrt n\,
    \left|
      \bar Z_\theta-\bar A_\theta
    \right|
    \ge
    C\frac{\rho_k\log n}{\sqrt n}
  \right\}
  \le
  c_1n^{-c_2},
  \label{eq:S-uniform-Z-bound}
\end{align}}
where
\[
  Z_{i,\theta}
  :=
  \frac12
  \log\frac{
    p_{\theta,i}(Y_i)+p_i^0(Y_i)
  }{
    2p_i^0(Y_i)
  },
  \qquad
  \bar Z_\theta
  :=
  \frac1n\sum_{i=1}^n Z_{i,\theta},
\]
and
\[
  \bar A_\theta
  :=
  \frac1n\sum_{i=1}^n
  \mathbb E_0 Z_{i,\theta}.
\]

Recall also that, with probability at least $1-c_1n^{-c_2}$,
\[
  \overline h_n\!\left(
    \bar p_{\widehat\theta}^{(n)},p_0^{(n)}
  \right)
  \le
  C\left(
    \frac{\rho_k\log n}{n}
  \right)^{1/2}
  =
  R.
\]
Combining this event with \eqref{eq:S-uniform-Z-bound} gives
\begin{align}
  \mathbb P_0\left\{
    \left|
      \bar Z_{\widehat\theta}
      -
      \bar A_{\widehat\theta}
    \right|
    \ge
    C\frac{\rho_k\log n}{n}
  \right\}
  \le
  2c_1n^{-c_2}.
  \label{eq:S-fitted-Z-bound}
\end{align}

Furthermore, by non-negativity of the Kullback--Leibler divergence,
\begin{align}
  \bar A_{\widehat\theta}
  =
  \frac{1}{2n}
  \sum_{i=1}^n
  \mathbb E_0
  \log\frac{
    \bar p_{\widehat\theta,i}(Y_i)
  }{
    p_i^0(Y_i)
  }
  =
  -\frac{1}{2n}
  \sum_{i=1}^n
  \mathrm{KL}\!\left(
    p_i^0
    \,\middle\|\,
    \bar p_{\widehat\theta,i}
  \right)
  \le 0.
  \label{eq:S-midpoint-KL}
\end{align}
It follows from
\eqref{eq:S-fitted-Z-bound} and
\eqref{eq:S-midpoint-KL} that
\begin{align}
  \mathbb P_0\left\{
    \bar Z_{\widehat\theta}
    \le
    C\frac{\rho_k\log n}{n}
  \right\}
  \ge
  1-2c_1n^{-c_2}.
  \label{eq:S-fitted-Z-upper}
\end{align}

By \eqref{eq:Jensen_ineq},
\[
  \frac12
  \log\frac{
    p_{\widehat\theta,i}(Y_i)
  }{
    p_i^0(Y_i)
  }
  \le
  \log\frac{
    \bar p_{\widehat\theta,i}(Y_i)
  }{
    p_i^0(Y_i)
  }
  =
  2Z_{i,\widehat\theta}.
\]
Therefore, on the event in \eqref{eq:S-fitted-Z-upper},
\begin{align*}
  &\overline{\mathcal L}_n\!\left(
    \widehat f_n^{(k)},
    \widehat\sigma_{n,k}^2
  \right)
  -
  \overline{\mathcal L}_n(f^0,\sigma_0^2)
  =
  \frac1n\sum_{i=1}^n
  \log\frac{
    p_{\widehat\theta,i}(Y_i)
  }{
    p_i^0(Y_i)
  }
  \le
  \frac4n\sum_{i=1}^n
  Z_{i,\widehat\theta}
  =
  4\bar Z_{\widehat\theta}
  \le
  C\frac{\rho_k\log n}{n},
\end{align*}
where the numerical factor is absorbed into $C$.

Thus, for every fixed $k\ge k_0$,
\[
  \mathbb P_0\left\{
    \overline{\mathcal L}_n\!\left(
      \widehat f_n^{(k)},
      \widehat\sigma_{n,k}^2
    \right)
    -
    \overline{\mathcal L}_n(f^0,\sigma_0^2)
    \le
    C\frac{\rho_k\log n}{n}
  \right\}
  \ge
  1-2c_1n^{-c_2}.
\]
Finally,
\[
  \overline{\mathcal L}_n\!\left(
    \widehat f_n^{(k)},
    \widehat\sigma_{n,k}^2
  \right)
  -
  \overline{\mathcal L}_n(f^0,\sigma_0^2)
  \ge0
\]
follows from the MLE property, because
$f^0\in\mathcal F_k$ for every $k\ge k_0$.
\end{proof}
\section{Proof of Lemma~3: Local geometric approximation bounds}
\label{sec:S-proof-geometry}

We prove the two assertions of Lemma~3 separately. Both are deterministic
approximation results and do not depend on the estimator, the sample size,
or the event that the correct number of segments is selected.

The first result gives the exact integrated squared error of the best affine
approximation to a continuous piecewise-linear function across a true change
in velocity. It is used both in the underfitting argument and in the first
stage of the changepoint-localization proof.

\subsection{The cubic one-segment approximation bound}
\label{subsec:S-cubic-bound}

\begin{lemma}[Best affine approximation across a change in velocity]
\label{lem:one-segment-approximation}
Let $\tau\in(0,1)$ and let $\epsilon>0$ be such that
$[\tau-\epsilon,\tau+\epsilon]\subset[0,1]$. Suppose that
$f^0:[\tau-\epsilon,\tau+\epsilon]\to\Rbb^d$ is continuous and
piecewise linear, with
$$
f^0(\tau+s)
=
\begin{cases}
f^0(\tau)+V_-s,
& -\epsilon\leq s\leq 0,\\[1mm]
f^0(\tau)+V_+s,
& 0\leq s\leq\epsilon,
\end{cases}
$$
where $V_-,V_+\in\Rbb^d$.

Let $\mathcal L_d(\tau,\epsilon)$ denote the class of all
$\Rbb^d$-valued affine functions on
$[\tau-\epsilon,\tau+\epsilon]$. Then
$$
\inf_{\ell\in\mathcal L_d(\tau,\epsilon)}
\int_{\tau-\epsilon}^{\tau+\epsilon}
\|\ell(t)-f^0(t)\|_2^2\,dt
=
\frac{\epsilon^3}{24}
\|V_+-V_-\|_2^2.
$$

In particular, if
$\|V_+-V_-\|_2\geq\underline C_2$, then
\begin{equation}
    \inf_{\ell\in\mathcal L_d(\tau,\epsilon)}
\int_{\tau-\epsilon}^{\tau+\epsilon}
\|\ell(t)-f^0(t)\|_2^2\,dt
\geq
\frac{\epsilon^3\underline C_2^2}{24}.\label{eq:S-cubic-lower-bound-C2}   
\end{equation}

\end{lemma}

\begin{proof}
Every $\ell\in\mathcal L_d(\tau,\epsilon)$ can be written uniquely as
$$
\ell(\tau+s)
=
f^0(\tau)+a+vs
$$
for some $a,v\in\Rbb^d$. Here,
$a=\ell(\tau)-f^0(\tau)$ is the approximation error at $\tau$, while
$v$ is the slope of $\ell$.

Define
$$
\Delta V:=V_+-V_-.
$$
Using the local representation of $f^0$, the integrated squared error is
$$
\begin{aligned}
\mathrm{Loss}(a,v)
&:=
\int_{\tau-\epsilon}^{\tau+\epsilon}
\|\ell(t)-f^0(t)\|_2^2\,dt=
\int_{-\epsilon}^{0}
\|a+(v-V_-)s\|_2^2\,ds+
\int_{0}^{\epsilon}
\|a+(v-V_+)s\|_2^2\,ds.
\end{aligned}
$$

Expanding and integrating gives
$$
\begin{aligned}
\mathrm{Loss}(a,v)
&=
2\epsilon\|a\|_2^2
-
\epsilon^2\langle a,\Delta V\rangle+
\frac{\epsilon^3}{3}
\left\{
\|v-V_-\|_2^2
+
\|v-V_+\|_2^2
\right\}.
\end{aligned}
$$

The dependence on $a$ and $v$ separates in this expression. Differentiating
with respect to $a$ gives
$$
4\epsilon a-\epsilon^2\Delta V=0,
$$
and therefore
$$
a^\star
=
\frac{\epsilon}{4}\Delta V.
$$

Differentiating with respect to $v$ gives
$$
2(v-V_-)+2(v-V_+)=0,
$$
and hence
$$
v^\star
=
\frac{V_-+V_+}{2}.
$$

Since $\mathrm{Loss}(a,v)$ is strictly convex in $(a,v)$, these values
define the unique minimizer. Thus, the best affine approximation is
$$
\ell^\star(\tau+s)
=
f^0(\tau)
+
\frac{\epsilon}{4}(V_+-V_-)
+
\frac{V_-+V_+}{2}s.
$$

At the minimizing values,
$$
v^\star-V_-
=
\frac{\Delta V}{2},
\qquad
v^\star-V_+
=
-\frac{\Delta V}{2}.
$$
Substituting $a^\star$ and $v^\star$ into the loss yields
$$
\begin{aligned}
\mathrm{Loss}(a^\star,v^\star)
&=
2\epsilon
\left\|
\frac{\epsilon}{4}\Delta V
\right\|_2^2
-
\epsilon^2
\left\langle
\frac{\epsilon}{4}\Delta V,
\Delta V
\right\rangle+
\frac{\epsilon^3}{3}
\left\{
\left\|
\frac{\Delta V}{2}
\right\|_2^2
+
\left\|
-\frac{\Delta V}{2}
\right\|_2^2
\right\}\\
&=
\left(
\frac18-\frac14+\frac16
\right)
\epsilon^3\|\Delta V\|_2^2=
\frac{\epsilon^3}{24}
\|\Delta V\|_2^2.
\end{aligned}
$$

Therefore,
$$
\inf_{\ell\in\mathcal L_d(\tau,\epsilon)}
\int_{\tau-\epsilon}^{\tau+\epsilon}
\|\ell(t)-f^0(t)\|_2^2\,dt
=
\frac{\epsilon^3}{24}
\|V_+-V_-\|_2^2.
$$
The final inequality follows immediately from
$\|V_+-V_-\|_2\geq\underline C_2$.
\end{proof}

To obtain the first assertion of Lemma~3, apply
Lemma~\ref{lem:one-segment-approximation} with
$$
\tau=\tau_j^0,
\qquad
V_-=V_j^0,
\qquad
V_+=V_{j+1}^0.
$$
For every
$$
0<\epsilon<
\frac12
\min\left\{
\tau_j^0-\tau_{j-1}^0,\,
\tau_{j+1}^0-\tau_j^0
\right\},
$$
the interval
$[\tau_j^0-\epsilon,\tau_j^0+\epsilon]$ contains no other true
changepoint. Hence,
$$
\begin{aligned}
\inf_{\ell\in\mathcal L_d}
\int_{\tau_j^0-\epsilon}^{\tau_j^0+\epsilon}
\|f^0(t)-\ell(t)\|_2^2\,dt
&=
\frac{\epsilon^3}{24}
\|V_{j+1}^0-V_j^0\|_2^2\geq
\frac{\underline C_2^2}{24}\epsilon^3.
\end{aligned}
$$
\subsection{Quadratic two-segment approximation bound}
\label{subsec:S-quadratic-bound}

We now prove the second assertion of
Lemma~3. Fix a true changepoint $\tau_j^0$ and
fixed points satisfying
$$
\frac{\tau_{j-1}^0+\tau_j^0}{2}
<t_1<\tau_j^0<t_2<
\frac{\tau_j^0+\tau_{j+1}^0}{2}.
$$
On $[t_1,\tau_j^0]$, the true signal has velocity $V_j^0$, while on
$[\tau_j^0,t_2]$, it has velocity $V_{j+1}^0$.

We use the following elementary identity. If
$r:[u,v]\to\Rbb^d$ is affine, then
$$
\int_u^v\|r(t)\|_2^2\,dt
=
\frac{v-u}{3}
\left\{
\|r(u)\|_2^2
+
\langle r(u),r(v)\rangle
+
\|r(v)\|_2^2
\right\}.
$$
Since
$$
\|p\|_2^2+\langle p,q\rangle+\|q\|_2^2
\geq
\frac12
\left(
\|p\|_2^2+\|q\|_2^2
\right),
$$
it follows that
\begin{equation}
\int_u^v\|r(t)\|_2^2\,dt
\geq
\frac{v-u}{6}
\left\{
\|r(u)\|_2^2+\|r(v)\|_2^2
\right\}.
\label{eq:affine-endpoint-lower-bound}
\end{equation}

Let $\tau\in(t_1,t_2)$ and
$g\in\mathcal G_2(\tau;t_1,t_2)$. By symmetry of the argument, it is
enough to consider the case
$
\tau\leq\tau_j^0.
$
The case $\tau\geq\tau_j^0$ follows by interchanging the roles of the
left and right subintervals and of the velocities
$V_j^0$ and $V_{j+1}^0$.

Set
$
\eta:=\tau_j^0-\tau,
a:=\tau-t_1,
b:=t_2-\tau_j^0.
$
Define the residual
$
\Delta g:=g-f^0
$
and write
$
x:=\Delta g(t_1),
y:=\Delta g(\tau),
z:=\Delta g(\tau_j^0),
w:=\Delta g(t_2).
$

The residual $\Delta g$ is affine on each of the intervals
$[t_1,\tau]$, $[\tau,\tau_j^0]$, and
$[\tau_j^0,t_2]$. Therefore,
\eqref{eq:affine-endpoint-lower-bound} gives
$$
\begin{aligned}
\int_{t_1}^{t_2}
\|\Delta g(t)\|_2^2\,dt
&\geq
\frac{a}{6}
\left(
\|x\|_2^2+\|y\|_2^2
\right)
+
\frac{\eta}{6}
\left(
\|y\|_2^2+\|z\|_2^2
\right)\\
&\quad+
\frac{b}{6}
\left(
\|z\|_2^2+\|w\|_2^2
\right).
\end{aligned}
$$

Choose $\delta_j>0$ such that
$
\delta_j
\leq
\frac12
\min\left\{
\tau_j^0-t_1,\,
t_2-\tau_j^0
\right\}.
$
Whenever $0\leq\eta\leq\delta_j$,
$$
a
=
\tau_j^0-t_1-\eta
\geq
\frac12(\tau_j^0-t_1).
$$
Hence, there exists a constant $c_{0,j}>0$, independent of $\tau$ and
$g$, such that
\begin{equation}
\int_{t_1}^{t_2}
\|\Delta g(t)\|_2^2\,dt
\geq
c_{0,j}
\left(
\|y\|_2^2+\|z\|_2^2+\|w\|_2^2
\right).
\label{eq:quadratic-loss-residual-bound}
\end{equation}

Because $g$ has a single velocity on $[\tau,t_2]$, let $v_g$ denote
that velocity. On $[\tau,\tau_j^0]$,
$
\dfrac{z-y}{\eta}
=
v_g-V_j^0,
$
whereas on $[\tau_j^0,t_2]$,
$
\dfrac{w-z}{b}
=
v_g-V_{j+1}^0.
$
Subtracting gives
$
\dfrac{z-y}{\eta}
-
\dfrac{w-z}{b}
=
V_{j+1}^0-V_j^0.
$
Equivalently,
$
(b+\eta)z-by-\eta w
=
\eta b\left(V_{j+1}^0-V_j^0\right).
$

By the Cauchy--Schwarz inequality,
$$
\begin{aligned}
\eta^2b^2
\|V_{j+1}^0-V_j^0\|_2^2
&=
\|(b+\eta)z-by-\eta w\|_2^2\\
&\leq
\left\{
(b+\eta)^2+b^2+\eta^2
\right\}
\left(
\|y\|_2^2+\|z\|_2^2+\|w\|_2^2
\right).
\end{aligned}
$$
Since $\eta\leq\delta_j\leq b/2$,
$$
(b+\eta)^2+b^2+\eta^2
\leq
4b^2.
$$
Therefore,
$$
\|y\|_2^2+\|z\|_2^2+\|w\|_2^2
\geq
\frac{\eta^2}{4}
\|V_{j+1}^0-V_j^0\|_2^2.
$$

Combining this with
\eqref{eq:quadratic-loss-residual-bound} yields
$$
\int_{t_1}^{t_2}
\|g(t)-f^0(t)\|_2^2\,dt
\geq
\frac{c_{0,j}}{4}
|\tau-\tau_j^0|^2
\|V_{j+1}^0-V_j^0\|_2^2.
$$

The same bound, possibly with a different constant, holds when
$\tau\geq\tau_j^0$ by the symmetry described above. Redefining $c_j>0$
as the smaller of the two constants, we obtain
$$
\inf_{g\in\mathcal G_2(\tau;t_1,t_2)}
\int_{t_1}^{t_2}
\|f^0(t)-g(t)\|_2^2\,dt
\geq
c_j
|\tau-\tau_j^0|^2
\|V_{j+1}^0-V_j^0\|_2^2
$$
whenever $|\tau-\tau_j^0|\leq\delta_j$.

Finally, by the minimum-jump condition,
$$
\|V_{j+1}^0-V_j^0\|_2
\geq
\underline C_2,
$$
and therefore
$$
\inf_{g\in\mathcal G_2(\tau;t_1,t_2)}
\int_{t_1}^{t_2}
\|f^0(t)-g(t)\|_2^2\,dt
\geq
c_j\underline C_2^2
|\tau-\tau_j^0|^2.
$$
This proves the second assertion of
Lemma~3.
\section{Proof of Theorem~4: Consistency of the penalized CPL estimator}

We first record an auxiliary empirical-process bound that will be used
in the proof of Theorem~4. It controls uniformly, over all candidate
piecewise-linear signals with at most $\overline k$ segments, the
cross-term between the Gaussian noise and the signal approximation
error.

\label{sec:S-proof-consistency}

\begin{lemma}[Uniform control of the noise cross-term]
\label{lem:S-uniform-cross-term}

Suppose that the initial-position and velocity parameter spaces are
compact, that $\overline k<\infty$ is fixed, and that the candidate
changepoints belong to the observation grid
$\{t_1,\dots,t_n\}$. Let
$$
\varepsilon_i
\overset{\mathrm{ind.}}{\sim}
\mathcal N_d(0,\sigma_0^2I_d),
\qquad i=1,\dots,n.
$$
Then
$$
\sup_{1\leq k\leq\overline k}
\sup_{f\in\mathcal F_k}
\left|
\frac1n
\sum_{i=1}^n
\left\langle
\varepsilon_i,
f(t_i)-f^0(t_i)
\right\rangle
\right|
=
O_{\Pbb}\left(
\sqrt{\frac{\log n}{n}}
\right).
$$
\end{lemma}

\begin{proof}

For $f\in\mathcal F_k$, define
$$
Z_n(f)
:=
\frac1n
\sum_{i=1}^n
\left\langle
\varepsilon_i,
f(t_i)-f^0(t_i)
\right\rangle.
$$

Because the initial-position and velocity parameter spaces are compact
and $\overline k$ is fixed, there exists a constant $M<\infty$ such
that
\begin{equation}
\label{eq:S-uniform-signal-bound}
\sup_{1\leq k\leq\overline k}
\sup_{f\in\mathcal F_k}
\sup_{t\in[0,1]}
\|f(t)-f^0(t)\|_2
\leq M.
\end{equation}

Fix $k\leq\overline k$ and a particular configuration of the
$k-1$ changepoints. Conditional on this configuration, a signal
$f\in\mathcal F_k$ is indexed by a finite-dimensional parameter vector
consisting of its initial position and segment velocities. Denote the
dimension of this parameter vector by $\rho_k$. Compactness of the
parameter space and the Lipschitz dependence of $f$ on its parameters
imply that, for every $\eta>0$, this fixed-configuration class admits an
$\eta$-net in the supremum norm with cardinality at most
$
C\eta^{-\rho_k},
$
where $C$ does not depend on the changepoint configuration.

There are at most $n^{k-1}$ possible configurations of the
$k-1$ changepoints. Taking $\eta_n=n^{-2}$ and combining the nets over
all configurations and all $k\leq\overline k$, we obtain a finite set
$\mathcal N_n$ such that
$
|\mathcal N_n|
\leq
C
\sum_{k=1}^{\overline k}
n^{k-1}n^{2\rho_k}
\leq
Cn^A
$
for some constants $C,A<\infty$, and such that, for every
$f\in\bigcup_{k\leq\overline k}\mathcal F_k$, there exists
$\widetilde f\in\mathcal N_n$ satisfying
\begin{equation}
\label{eq:S-cross-term-net-approximation}
\max_{1\leq i\leq n}
\|f(t_i)-\widetilde f(t_i)\|_2
\leq
n^{-2}.
\end{equation}

For each fixed $\widetilde f\in\mathcal N_n$, the random variable
$Z_n(\widetilde f)$ is centered Gaussian. By
\eqref{eq:S-uniform-signal-bound},
$$
\begin{aligned}
\operatorname{Var}_0\{Z_n(\widetilde f)\}
&=
\frac{\sigma_0^2}{n^2}
\sum_{i=1}^n
\|\widetilde f(t_i)-f^0(t_i)\|_2^2\leq
\frac{\sigma_0^2M^2}{n}.
\end{aligned}
$$
Therefore, for every $x>0$,
$
\Pbb_0\left(
|Z_n(\widetilde f)|>x
\right)
\leq
2\exp\left(
-\dfrac{nx^2}{2\sigma_0^2M^2}
\right).
$
A union bound over $\mathcal N_n$ gives
$$
\begin{aligned}
\Pbb_0\left(
\max_{\widetilde f\in\mathcal N_n}
|Z_n(\widetilde f)|>x
\right)
&\leq
2|\mathcal N_n|
\exp\left(
-\frac{nx^2}{2\sigma_0^2M^2}
\right)\leq
2Cn^A
\exp\left(
-\frac{nx^2}{2\sigma_0^2M^2}
\right).
\end{aligned}
$$
Taking
$
x=K\sqrt{\frac{\log n}{n}}
$
with $K$ sufficiently large shows that
\begin{equation}
\label{eq:S-cross-term-net-bound}
\max_{\widetilde f\in\mathcal N_n}
|Z_n(\widetilde f)|
=
O_{\Pbb}\left(
\sqrt{\frac{\log n}{n}}
\right).
\end{equation}

It remains to control the approximation error between an arbitrary
$f$ and its net representative $\widetilde f$. By
\eqref{eq:S-cross-term-net-approximation},
$$
\begin{aligned}
|Z_n(f)-Z_n(\widetilde f)|
&\leq
\frac1n
\sum_{i=1}^n
\|\varepsilon_i\|_2
\,
\|f(t_i)-\widetilde f(t_i)\|_2\leq
n^{-2}
\left(
\frac1n\sum_{i=1}^n\|\varepsilon_i\|_2
\right).
\end{aligned}
$$
Since
$
\frac1n\sum_{i=1}^n\|\varepsilon_i\|_2
=
O_{\Pbb}(1),
$
we have
$$
\sup_{1\leq k\leq\overline k}
\sup_{f\in\mathcal F_k}
\inf_{\widetilde f\in\mathcal N_n}
|Z_n(f)-Z_n(\widetilde f)|
=
O_{\Pbb}(n^{-2}).
$$
Combining this with
\eqref{eq:S-cross-term-net-bound} proves that
$$
\sup_{1\leq k\leq\overline k}
\sup_{f\in\mathcal F_k}
|Z_n(f)|
=
O_{\Pbb}\left(
\sqrt{\frac{\log n}{n}}
\right).
$$
\end{proof}

\begin{restatedtheorem}[Theorem~4 of the main text]
Suppose the conditions of Lemma~2 hold, the minimum-spacing and minimum-jump
conditions in the main text are satisfied, $\overline k\ge k_0$, and
$\gamma>1$. Then the penalized estimator satisfies
\[
  \Pbb\!\left(\widehat k_n=k_0,\ \max_{1\le j\le k_0-1}
  |\widehat\tau_{n,j}-\tau_j^0|\le C(\log n/n)^{1/2}\right)\to1.
\]
\end{restatedtheorem}

\begin{proof} Our estimate for the number of changepoints is given by
$$\widehat{k}_n = \argmax_{k\leq \overline{k}} \left[2\overline{\mathcal{L}}_n(\widehat{f}_n^{(k)}, \widehat{\sigma}^2_{n, k}) - \rho_k \dfrac{(\log(n))^{\gamma}}{n}\right],$$
where $\rho_k=d(k+1)+k$ is the number of parameters of the model.

We aim to prove that $$\Pbb(\widehat{k}_n = k_0) \to 1$$ by first showing that $\Pbb(\widehat{k}_n > k_0) \to 0$ and then $\Pbb(\widehat{k}_n < k_0) \to 0$ as $n\to \infty$.
\paragraph{Asymptotically, sSIC does not over-fit.}
From Lemma~\ref{lem:S-rates}, we have with probability tending to 1, for all $k\geq k_0$,
$$0\leq \overline{\mathcal{L}}_n(\widehat{f}_n^{(k)}, \widehat{\sigma}^2_{n, k}) - \overline{\mathcal{L}}_n({f}^0, \sigma_0^2) \leq C\rho_k\dfrac{\log(n)}{n},$$
for some positive constant $C$ depending on $\overline{k}$ and the parameter space. Therefore, for all $k > k_0$ (and $k < \overline{k})$, the increase in average log-likelihood when over-fitting can be characterized by
\begin{align*}
    2\overline{\mathcal{L}}_n(\widehat{f}_n^{(k)}, \widehat{\sigma}^2_{n, k}) - 2\overline{\mathcal{L}}_n(\widehat{f}_n^{(k_0)}, \widehat{\sigma}^2_{n, k_0}) & \leq  2\overline{\mathcal{L}}_n(\widehat{f}_n^{(k)}, \widehat{\sigma}^2_{n, k}) - 2\overline{\mathcal{L}}_n({f}^0, \sigma_0^2) \\
    & \leq 2C \rho_k \dfrac{\log(n)}{n} = 2C [d(k+1)+k] \dfrac{\log(n)}{n} \\
    & < (d+1)(k - k_0) \dfrac{(\log(n))^{\gamma}}{n},
\end{align*}
for all $n$ large enough, since $\gamma>1$ and $d$ and $\overline{k}$ are fixed, so the constants can be chosen uniformly over $k_0\le k\le\overline{k}$. Hence,
\begin{align*}
& 2\overline{\mathcal{L}}_n(\widehat{f}_n^{(k_0)}, \widehat{\sigma}^2_{n, k_0}) - \rho_{k_0} \dfrac{(\log(n))^{\gamma}}{n} \\
& > 2\overline{\mathcal{L}}_n(\widehat{f}_n^{(k)}, \widehat{\sigma}^2_{n, k}) - \rho_k \dfrac{(\log(n))^{\gamma}}{n}\quad \forall k\in [k_0 + 1, \overline{k}],
\end{align*}
for all sufficiently large $n$, implying that
$$\Pbb(\widehat{k}_n > k_0) \to 0,$$
as $n\to \infty$.

\paragraph{Asymptotically, sSIC does not underfit.}

For $f\in\mathcal F_k$, define the normalized residual sum of squares
$
R_n(f)
:=
\frac{1}{dn}
\sum_{i=1}^n
\|Y_i-f(t_i)\|_2^2,
$
and write
$
\widehat R_{n,k}
:=
R_n\left(\widehat f_n^{(k)}\right).
$

We first establish a uniform positive approximation gap for every
underfitted model. Let
$
\epsilon_0:=\dfrac{\underline C_1}{4},
$
and, for each true changepoint $\tau_j^0$, define
$$
I_j
:=
[\tau_j^0-\epsilon_0,\tau_j^0+\epsilon_0],
\qquad
j=1,\dots,k_0-1.
$$
By the minimum-spacing condition, these intervals are pairwise disjoint
and each contains exactly one true changepoint.

Fix $k<k_0$ and $f\in\mathcal F_k$. The function $f$ has at most
$k-1<k_0-1$ changepoints. Therefore, at least one interval $I_j$
contains no changepoint of $f$, so that $f$ is affine on $I_j$.
By the cubic one-segment approximation bound in Lemma~3 (i.e., 
Lemma~\ref{lem:one-segment-approximation}),
$$
\begin{aligned}
\int_0^1
\|f(t)-f^0(t)\|_2^2\,dt
&\geq
\int_{I_j}
\|f(t)-f^0(t)\|_2^2\,dt\geq
\frac{\epsilon_0^3}{24}
\|V_{j+1}^0-V_j^0\|_2^2 \geq
\frac{\epsilon_0^3\underline C_2^2}{24}.
\end{aligned}
$$
Consequently, defining
$
\Delta_0
:=
\dfrac{\epsilon_0^3\underline C_2^2}{24}
>0,
$
we have
\begin{equation}
\label{eq:S-underfit-population-gap}
\inf_{1\leq k<k_0}
\inf_{f\in\mathcal F_k}
\int_0^1
\|f(t)-f^0(t)\|_2^2\,dt
\geq
\Delta_0.
\end{equation}

Because the initial-position and velocity parameter spaces are compact,
there exist constants $M,L<\infty$, independent of
$f\in\bigcup_{k<k_0}\mathcal F_k$, such that
$
\sup_{t\in[0,1]}
\|f(t)-f^0(t)\|_2
\leq M
$
and
$
\|(f-f^0)(t)-(f-f^0)(s)\|_2
\leq
L|t-s|.
$
Hence, the functions
$
q_f(t):=\|f(t)-f^0(t)\|_2^2
$
are uniformly Lipschitz. Since $t_i=i/n$, their Riemann-sum errors
satisfy
$$
\sup_{1\leq k<k_0}
\sup_{f\in\mathcal F_k}
\left|
\|f-f^0\|_n^2
-
\int_0^1
\|f(t)-f^0(t)\|_2^2\,dt
\right|
\leq
\frac{C}{n}
$$
for some constant $C<\infty$. We have, for every $1\leq k<k_0$ and every $f\in\mathcal F_k$,
$$
\|f-f^0\|_n^2
\geq
\int_0^1
\|f(t)-f^0(t)\|_2^2\,dt
-
\frac{C}{n}.
$$
Taking the infimum over $f\in\mathcal F_k$ and then over
$1\leq k<k_0$ gives
$$
\begin{aligned}
\inf_{1\leq k<k_0}
\inf_{f\in\mathcal F_k}
\|f-f^0\|_n^2
&\geq
\inf_{1\leq k<k_0}
\inf_{f\in\mathcal F_k}
\left\{
\int_0^1
\|f(t)-f^0(t)\|_2^2\,dt
-
\frac{C}{n}
\right\}\\
&=
\inf_{1\leq k<k_0}
\inf_{f\in\mathcal F_k}
\int_0^1
\|f(t)-f^0(t)\|_2^2\,dt
-
\frac{C}{n}\\
&\overset{\eqref{eq:S-underfit-population-gap}}{\geq}
\Delta_0-\frac{C}{n},
\end{aligned}
$$
where the equality follows because $C/n$ does not depend on $k$ or
$f$. Since $\Delta_0>0$, for all sufficiently large $n$,
$C/n\leq\Delta_0/2$. Therefore,
\begin{equation}
    \inf_{1\leq k<k_0}
\inf_{f\in\mathcal F_k}
\|f-f^0\|_n^2
\geq
\frac{\Delta_0}{2}.\label{eq:S-underfit-empirical-gap}
\end{equation}

Write
$$
Y_i=f^0(t_i)+\varepsilon_i,
\qquad
\varepsilon_i
\overset{\mathrm{ind.}}{\sim}
\mathcal N_d(0,\sigma_0^2I_d).
$$
For any $f\in\mathcal F_k$,
$$
\begin{aligned}
R_n(f)
&=
\frac{1}{dn}
\sum_{i=1}^n
\|\varepsilon_i\|_2^2
+
\frac{1}{d}\|f-f^0\|_n^2-
\frac{2}{dn}
\sum_{i=1}^n
\left\langle
\varepsilon_i,
f(t_i)-f^0(t_i)
\right\rangle.
\end{aligned}
$$

By Lemma~\ref{lem:S-uniform-cross-term},
\begin{equation}
\label{eq:S-underfit-cross-term}
\sup_{1\leq k<k_0}
\sup_{f\in\mathcal F_k}
\left|
\frac1n
\sum_{i=1}^n
\left\langle
\varepsilon_i,
f(t_i)-f^0(t_i)
\right\rangle
\right|
=
O_{\Pbb}\left(
\sqrt{\frac{\log n}{n}}
\right)
=
o_{\Pbb}(1).
\end{equation}

By the law of large numbers,
$
(dn)^{-1}
\sum_{i=1}^n
\|\varepsilon_i\|_2^2
\longrightarrow
\sigma_0^2
$
in probability. Combining this fact with
\eqref{eq:S-underfit-empirical-gap} and
\eqref{eq:S-underfit-cross-term} gives
$$
\begin{aligned}
\inf_{1\leq k<k_0}\widehat R_{n,k}
&\geq
\frac{1}{dn}\sum_{i=1}^n\|\varepsilon_i\|_2^2
+
\frac{\Delta_0}{2d}-
\frac{2}{d}
\sup_{1\leq k<k_0}
\sup_{f\in\mathcal F_k}
\left|
\frac1n\sum_{i=1}^n
\left\langle
\varepsilon_i,f(t_i)-f^0(t_i)
\right\rangle
\right|\\
&=
\sigma_0^2+\frac{\Delta_0}{2d}+o_{\Pbb}(1).
\end{aligned}
$$
Therefore, with probability tending to one,
$$
\inf_{1\leq k<k_0}\widehat R_{n,k}
\geq
\sigma_0^2+\frac{\Delta_0}{4d}.
$$
In contrast, Lemma~2
implies
$$
\widehat R_{n,k_0}
\longrightarrow
\sigma_0^2
$$
in probability.

It follows that there exists a constant $c_0>0$ such that, with
probability tending to one,
\begin{equation}
\label{eq:S-underfit-rss-separation}
\widehat R_{n,k}
\geq
\widehat R_{n,k_0}+c_0,
\qquad
k=1,\dots,k_0-1.
\end{equation}

For $s>0$, define the profiled Gaussian loss
$$
Q(s)
:=
\inf_{\sigma^2\in\Omega}
\left\{
d\log(\sigma^2)
+
\frac{ds}{\sigma^2}
\right\}.
$$
Since $\Omega$ is compact and bounded away from zero, $Q$ is continuous
and strictly increasing on $(0,\infty)$. Moreover,
\eqref{eq:S-underfit-rss-separation} and
$\widehat R_{n,k_0}\to\sigma_0^2$ in probability imply that there exists
a constant $c_1>0$ such that, with probability tending to one,
$$
Q(\widehat R_{n,k})
-
Q(\widehat R_{n,k_0})
\geq
c_1,
\qquad
k=1,\dots,k_0-1.
$$

On the other hand, because $\overline k$ is fixed,
$$
0
\leq
(\rho_{k_0}-\rho_k)
\frac{(\log n)^\gamma}{n}
\longrightarrow
0
$$
for every $k<k_0$. Thus, for all sufficiently large $n$, the fixed
positive profiled-likelihood gap dominates the smaller penalty assigned
to every underfitted model. Consequently,
$$
\Pbb(\widehat k_n<k_0)
\longrightarrow
0.
$$

Combining this result with the preceding overfitting argument gives
$$
\Pbb(\widehat k_n=k_0)
\longrightarrow
1.
$$

\paragraph{Convergence of the changepoints.}

We have
\begin{align*}
&\Pbb\left(
\widehat k_n=k_0,\,
\max_{1\leq i\leq k_0-1}
|\widehat\tau_i^n-\tau_i^0|
\leq
C_\tau
\left(\frac{\log n}{n}\right)^{1/2}
\right)
\\
&\qquad=
\Pbb\left(
\max_{1\leq i\leq k_0-1}
|\widehat\tau_i^n-\tau_i^0|
\leq
C_\tau
\left(\frac{\log n}{n}\right)^{1/2}
\,\middle|\,
\widehat k_n=k_0
\right)
\Pbb(\widehat k_n=k_0).
\end{align*}
Since we have already shown that
$$
\Pbb(\widehat k_n=k_0)\longrightarrow 1,
$$
it remains to prove that, for a sufficiently large constant
$C_\tau<\infty$,
\begin{equation}
\label{eq:qed-conv-cps}
\Pbb\left(
\max_{1\leq i\leq k_0-1}
|\widehat\tau_i^n-\tau_i^0|
\leq
C_\tau
\left(\frac{\log n}{n}\right)^{1/2}
\,\middle|\,
\widehat k_n=k_0
\right)
\longrightarrow 1.
\end{equation}

By Lemma~2, there exist constants $C_f,c_1,c_2>0$, independent of
$n$, such that
$$
\Pbb\left(
\|\widehat f_n^{(k_0)}-f^0\|_n^2
\leq
C_f\frac{\log n}{n}
\right)
\geq
1-c_1n^{-c_2}
$$
for all sufficiently large $n$. We will show that, on the event
$\{\widehat k_n=k_0\}$, this signal-estimation bound implies
$$
\max_{1\leq i\leq k_0-1}
|\widehat\tau_i^n-\tau_i^0|
\leq
C_\tau
\left(\frac{\log n}{n}\right)^{1/2}
$$
for a constant $C_\tau<\infty$ independent of $n$.

The argument proceeds in two steps. First, the cubic one-segment
approximation bound in Lemma~3 yields the preliminary localization rate
$$
\max_{1\leq i\leq k_0-1}
|\widehat\tau_i^n-\tau_i^0|
=
O_{\Pbb}\left\{
\left(\frac{\log n}{n}\right)^{1/3}
\right\}.
$$
This establishes consistency and ensures that each estimated
changepoint lies in a fixed neighborhood of its corresponding true
changepoint. Second, the quadratic two-segment approximation bound in
Lemma~3 uses the information on both sides of the true changepoint to
sharpen the rate to
$$
\max_{1\leq i\leq k_0-1}
|\widehat\tau_i^n-\tau_i^0|
=
O_{\Pbb}\left\{
\left(\frac{\log n}{n}\right)^{1/2}
\right\}.
$$
Figure~\ref{fig:convergence_changepoints} illustrates these two stages.

\begin{figure}
    \centering
    \includegraphics[width=1\linewidth]{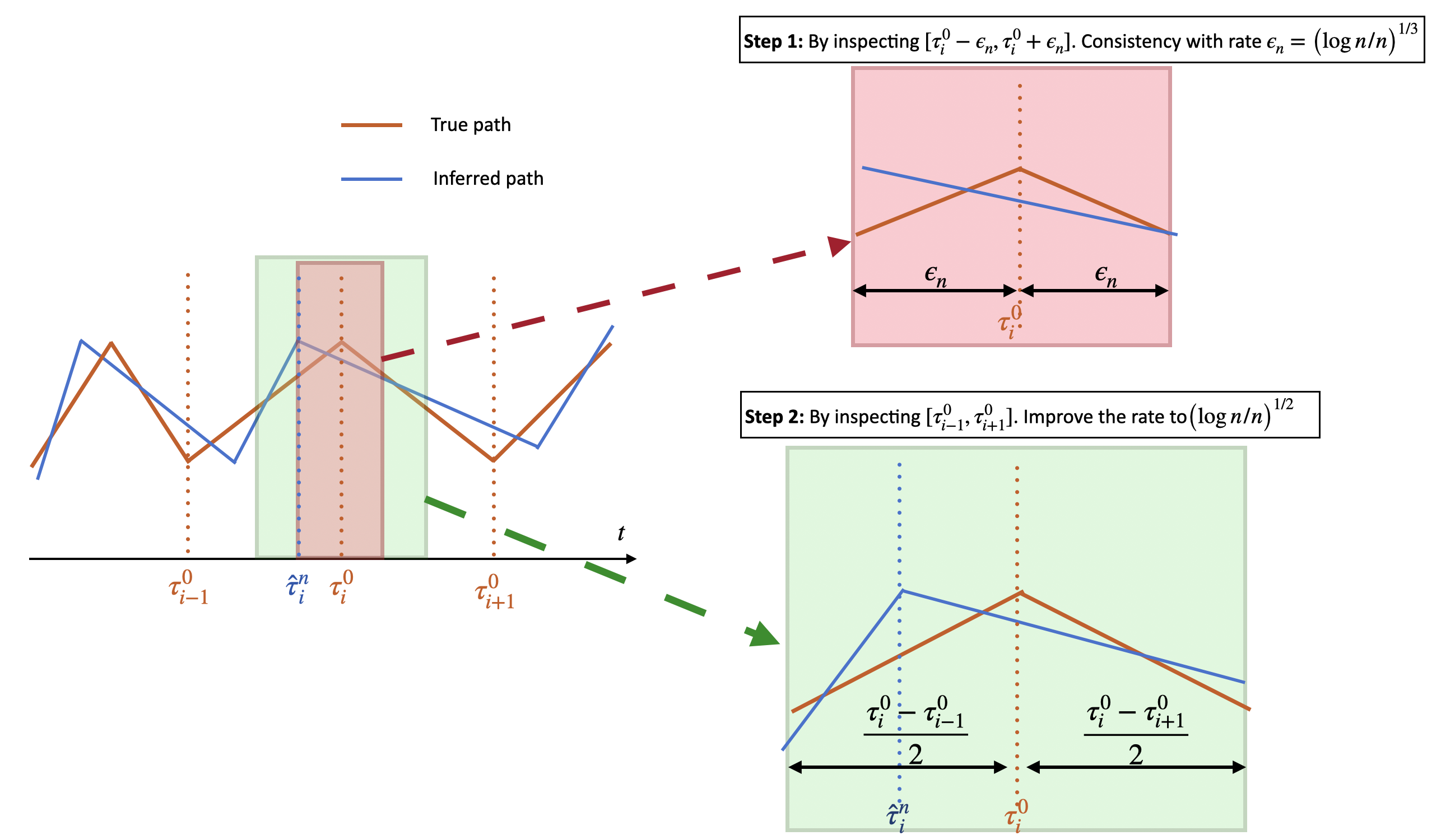}
    \caption{Illustration of the two-stage changepoint-localization argument.
The cubic one-segment approximation bound first establishes consistency
at rate $\{\log n/n\}^{1/3}$, after which the quadratic two-segment
bound sharpens the rate to $\{\log n/n\}^{1/2}$.}
    \label{fig:convergence_changepoints}
\end{figure}

\paragraph{Step 1. Consistency of the estimated changepoints.}

By Lemma~2,
$$
\|\widehat f_n^{(k_0)}-f^0\|_n^2
=
O_{\Pbb}\left(\frac{\log n}{n}\right).
$$
We combine this rate with the cubic one-segment approximation bound in
Lemma~3. We first show that every true changepoint has an estimated
changepoint within distance
$
O_{\Pbb}\left\{
\left(\frac{\log n}{n}\right)^{1/3}
\right\},
$
and then use the equality of the numbers of true and estimated
changepoints to obtain a one-to-one matching. This proves
$$
\widehat\tau_i^n\longrightarrow\tau_i^0,
\qquad
i=1,\dots,k_0-1,
$$
in probability. The sharper localization rate is established in
Step~2.
\paragraph{Step 1.1. Setup.} The main goal of Step 1 is to show
\begin{equation}\label{eq:rate-tau-1-3}
    |\widehat{\tau}_i^{n} - \tau_i^0| \lesssim \left(\dfrac{\log n}{n}\right)^{1/3} \quad \forall i\in [k_0-1],
\end{equation}
as $n\to \infty$, where the multiplicative constant in the inequality above does not depend on $n$.

We proceed with a proof by contradiction. Assume that 
this limit does not hold, then there exists an index $i\in[k_0-1]$ and an increasing subsequence $(m_n)_{n=1}^{\infty} \subset \mathbb{N}$ such that $\widehat{\tau}^{m_n}_j \not \in \left[\tau_i^0 - \epsilon_{m_n}, \tau_i^0 + \epsilon_{m_n} \right]$ for all $j=[k_0-1]$, with $\epsilon_{m_n} \gg (\log m_n/ m_n)^{1/3}$, i.e., $\lim_{n\to \infty} \dfrac{\epsilon_{m_n}}{(\log(m_n) / m_n)^{1/3}} $ $\to \infty$ as $n\to \infty$. Without loss of generality, we can further assume that $\tau_{i-1}^{0} < \tau_i^0 - \epsilon_{m_n} < \tau_i^0 + \epsilon_{m_n} < \tau_{i+1}^{0}$. Then,
$$\sum_{j = \lfloor m_n(\tau_i^0 - \epsilon_{m_n})\rfloor}^{\lfloor{m_n(\tau_i^0 + \epsilon_{m_n})}\rfloor} \norm{\widehat{f}_{m_n}^{(k_0)}(j/{m_n}) - f^0(j/{m_n})}^2 \leq m_n \norm{\widehat{f}_{m_n}^{(k_0)} - f^0}_{m_n}^2 \leq C\log(m_n).$$
We are inspecting the LHS to show the contradiction.
By the assumption, in the interval $\left[\lfloor m_n(\tau_i^0 - \epsilon_{m_n})\rfloor, \lfloor m_n(\tau_i^0 + \epsilon_{m_n})\rfloor \right]$, the function $\widehat{f}_{m_n}^{(k_0)}$ is a linear function, meanwhile $f^0$ is a piecewise linear function with 2 segments having slope $V_{i-1}^{0}$ and $V_{i}^{0}$, respectively. When $n$ is large enough, the Riemann-sum bound gives 
\begin{equation}
\begin{aligned}
\Bigg|
&\frac{1}{m_n}
\sum_{j=\lfloor m_n(\tau_i^0-\epsilon_{m_n})\rfloor}
^{\lfloor m_n(\tau_i^0+\epsilon_{m_n})\rfloor}
\left\|
\widehat f_{m_n}^{(k_0)}(j/m_n)-f^0(j/m_n)
\right\|_2^2-
\int_{\tau_i^0-\epsilon_{m_n}}^{\tau_i^0+\epsilon_{m_n}}
\left\|
\widehat f_{m_n}^{(k_0)}(t)-f^0(t)
\right\|_2^2\,dt
\Bigg|
\leq
\frac{C_R}{m_n}.
\label{eq:loss_function}
\end{aligned}
\end{equation}
Hence, we have
\begin{equation}\label{eq:rate-tau-integral}
    \int_{\tau_i^0 - \epsilon_{m_n}}^{\tau_i^0+\epsilon_{m_n}} \norm{\widehat{f}^{(k_0)}_{m_n}(t) - f^0(t)}^2 dt \leq C \dfrac{\log(m_n)}{m_n}.
\end{equation}

\paragraph{Step 1.2. Applying the best affine approximation bound.}

Our aim is to show that \eqref{eq:rate-tau-integral} cannot hold when
$$
\epsilon_{m_n}
\gg
\left\{
\frac{\log(m_n)}{m_n}
\right\}^{1/3}.
$$
By the assumption in Step~1.1,
$\widehat f_{m_n}^{(k_0)}$ has no changepoint in
$$
[\tau_i^0-\epsilon_{m_n},\tau_i^0+\epsilon_{m_n}],
$$
and is therefore affine on this interval.

The cubic one-segment approximation result in Lemma~3 gives
\begin{align}
\inf_{g\in\mathcal L_d}
\int_{\tau_i^0-\epsilon_{m_n}}^{\tau_i^0+\epsilon_{m_n}}
\|g(t)-f^0(t)\|_2^2\,dt
&=
\frac{\epsilon_{m_n}^3}{24}
\|V_{i+1}^0-V_i^0\|_2^2.
\label{eq:exact-best-linear-loss}
\end{align}
Here, $\mathcal L_d$ denotes the class of
$\Rbb^d$-valued affine functions on
$[\tau_i^0-\epsilon_{m_n},\tau_i^0+\epsilon_{m_n}]$.

Since $\widehat f_{m_n}^{(k_0)}$ belongs to this affine class on the
interval under consideration,
\begin{align}
\int_{\tau_i^0-\epsilon_{m_n}}^{\tau_i^0+\epsilon_{m_n}}
\left\|
\widehat f_{m_n}^{(k_0)}(t)-f^0(t)
\right\|_2^2\,dt\geq
\frac{\epsilon_{m_n}^3}{24}
\|V_{i+1}^0-V_i^0\|_2^2\geq
\frac{\underline C_2^2}{24}
\epsilon_{m_n}^3,
\label{eq:cubic-local-lower-bound}
\end{align}

where the second inequality follows from the minimum-jump condition.

\paragraph{Step 1.3. Deriving the contradiction.}

Combining the lower bound \eqref{eq:cubic-local-lower-bound} in Step 1.2 with \eqref{eq:rate-tau-integral} gives
$$
C\frac{\log(m_n)}{m_n}
\ge
\frac{\underline C_2^2}{24}\epsilon_{m_n}^3.
$$
Equivalently,
$$
\epsilon_{m_n}
\le
\left(\frac{24C}{\underline C_2^2}\right)^{1/3}
\left\{\frac{\log(m_n)}{m_n}\right\}^{1/3},
$$
which contradicts the assumption that
$\epsilon_{m_n}\gg\{\log(m_n)/m_n\}^{1/3}$.

Therefore,
$$
\max_{1\le i\le k_0-1}
|\widehat\tau_i^n-\tau_i^0|
=
O_{\Pbb}\!\left\{
\left(\frac{\log n}{n}\right)^{1/3}
\right\}
$$
on the event $\{\widehat k_n=k_0\}$.

\paragraph{Step 2. Improving the rate to
$\{\log n/n\}^{1/2}$.}

Step~1 established the preliminary localization rate
$$
\max_{1\leq i\leq k_0-1}
|\widehat\tau_i^n-\tau_i^0|
=
O_{\Pbb}\left\{
\left(\frac{\log n}{n}\right)^{1/3}
\right\}.
$$
We now use this consistency result together with the quadratic
two-segment approximation bound in Lemma~3 to obtain the sharper rate
$\{\log n/n\}^{1/2}$.

\paragraph{Step 2.1. Setup and an upper bound for the local loss.}

Fix $i\in\{1,\dots,k_0-1\}$ and choose fixed points $t_1,t_2$ such that
$$
\frac{\tau_{i-1}^0+\tau_i^0}{2}
<
t_1
<
\tau_i^0
<
t_2
<
\frac{\tau_i^0+\tau_{i+1}^0}{2}.
$$
By Step~1,
$
\widehat\tau_i^n
\longrightarrow
\tau_i^0
$
in probability. Moreover, with probability tending to one,
$\widehat\tau_i^n$ is the unique estimated changepoint in $[t_1,t_2]$.
On this event, both $f^0$ and
$\widehat f_n^{(k_0)}$ have exactly one changepoint in $[t_1,t_2]$,
located at $\tau_i^0$ and $\widehat\tau_i^n$, respectively.

By Lemma~2,
$
\|\widehat f_n^{(k_0)}-f^0\|_n^2
=
O_{\Pbb}\left(\dfrac{\log n}{n}\right).
$
Define
$
q_n(t)
:=
\|\widehat f_n^{(k_0)}(t)-f^0(t)\|_2^2.
$
Compactness of the initial-position and velocity parameter spaces
implies that the functions $q_n$ are uniformly Lipschitz. Hence, for
some constant $C_R<\infty$,
\begin{equation}
\label{eq:riemann-step2}
\left|
\frac1n
\sum_{j=\lceil nt_1\rceil}^{\lfloor nt_2\rfloor}
q_n(j/n)
-
\int_{t_1}^{t_2}q_n(t)\,dt
\right|
\leq
\frac{C_R}{n}.
\end{equation}
Therefore,
\begin{align}
\int_{t_1}^{t_2}
\|\widehat f_n^{(k_0)}(t)-f^0(t)\|_2^2\,dt
&\leq
\|\widehat f_n^{(k_0)}-f^0\|_n^2
+
\frac{C_R}{n}=
O_{\Pbb}\left(\frac{\log n}{n}\right).
\label{eq:step2-local-upper}
\end{align}

For $\tau\in(t_1,t_2)$, let
$
\mathcal G_2(\tau;t_1,t_2)
$
denote the class of continuous $\Rbb^d$-valued piecewise-linear
functions on $[t_1,t_2]$ having one changepoint at $\tau$.

Since
$
\widehat f_n^{(k_0)}|_{[t_1,t_2]}
\in
\mathcal G_2(\widehat\tau_i^n;t_1,t_2),
$
we have
\begin{align}
\inf_{g\in\mathcal G_2(\widehat\tau_i^n;t_1,t_2)}
\int_{t_1}^{t_2}
\|g(t)-f^0(t)\|_2^2\,dt
&\leq
\int_{t_1}^{t_2}
\|\widehat f_n^{(k_0)}(t)-f^0(t)\|_2^2\,dt
=
O_{\Pbb}\left(\frac{\log n}{n}\right).
\label{eq:step2-best-loss-upper}
\end{align}

\paragraph{Step 2.2. Applying the quadratic two-segment bound.}

By Step~1,
$
|\widehat\tau_i^n-\tau_i^0|\to0
$
in probability. Therefore, with probability tending to one,
$\widehat\tau_i^n$ lies in the neighborhood of $\tau_i^0$ on which the
quadratic two-segment approximation bound in Lemma~3 applies.

Consequently, there exist constants $c_i>0$ and $\delta_i>0$ such that,
whenever
$
|\widehat\tau_i^n-\tau_i^0|\leq\delta_i,
$
\begin{align}
\inf_{g\in\mathcal G_2(\widehat\tau_i^n;t_1,t_2)}
\int_{t_1}^{t_2}
\|g(t)-f^0(t)\|_2^2\,dt
&\geq
c_i
|\widehat\tau_i^n-\tau_i^0|^2
\|V_{i+1}^0-V_i^0\|_2^2
\nonumber\\
&\geq
c_i\underline C_2^2
|\widehat\tau_i^n-\tau_i^0|^2,
\label{eq:step2-quadratic-lower}
\end{align}
where the second inequality follows from the minimum-jump condition.

\paragraph{Step 2.3. Deriving the improved localization rate.}

Combining
\eqref{eq:step2-best-loss-upper} and
\eqref{eq:step2-quadratic-lower} gives
$$
c_i\underline C_2^2
|\widehat\tau_i^n-\tau_i^0|^2
=
O_{\Pbb}\left(\frac{\log n}{n}\right).
$$
Hence,
$$
|\widehat\tau_i^n-\tau_i^0|
=
O_{\Pbb}\left\{
\left(\frac{\log n}{n}\right)^{1/2}
\right\}.
$$

Since $k_0$ is fixed and there are only finitely many true
changepoints, the bound holds uniformly over
$i=1,\dots,k_0-1$. Therefore,
\begin{equation}
\label{eq:final-tau-rate}
\max_{1\leq i\leq k_0-1}
|\widehat\tau_i^n-\tau_i^0|
=
O_{\Pbb}\left\{
\left(\frac{\log n}{n}\right)^{1/2}
\right\}.
\end{equation}
\end{proof}

\bibliographystyle{plainnat}
\bibliography{bib}

\end{document}